\documentclass[11pt]{article}

\usepackage{geometry}
\usepackage{fancyhdr, titling}

\newcommand\arxivversion

\usepackage[nottoc]{tocbibind}

\usepackage{common_eBY}
\usepackage{common_macros}
\usepackage{arxiv_common_eBY}
\newcommand{\drafttitle}{Post-selection inference for e-value based confidence intervals}
\hypersetup{
    colorlinks=false,
    linkcolor=blue,
    filecolor=magenta,
    urlcolor=cyan,
    pdftitle={\drafttitle},
    pdfpagemode=FullScreen,
}

\hypersetup{pdfauthor={Ziyu Xu, Ruodu Wang, Aaditya Ramdas},
            pdftitle={\drafttitle}}
\title{\drafttitle}
\author{Ziyu Xu$^1$, Ruodu Wang$^3$, Aaditya Ramdas$^{1, 2}$\\
\newline
\and
Departments of $^1$Statistics and $^2$Machine Learning, Carnegie Mellon University\\
$^3$Department of Statistics and Actuarial Science, University of Waterloo\\
\newline
\and
\texttt{\{xzy,aramdas\}@cmu.edu, wang@uwaterloo.ca}}
\begin{document}
\maketitle

\begin{abstract}
    Suppose that one can construct a valid $(1-\delta)$-confidence interval (CI) for each of $K$ parameters of potential interest. If a data analyst uses an arbitrary data-dependent criterion to select some subset $S$ of parameters, then the aforementioned CIs for the selected parameters are no longer valid due to selection bias. We design a new method to adjust the intervals in order to control the false coverage rate (FCR).
The main established method is the ``BY procedure''  by Benjamini and Yekutieli (JASA, 2005). The BY guarantees require certain restrictions on the selection criterion and on the dependence between the CIs. We propose a new simple method which, in contrast, is valid under any dependence structure between the original CIs, and any (unknown) selection criterion, but which only applies to a special, yet broad, class of CIs that we call e-CIs. To elaborate, our procedure simply reports $(1-\delta|S|/K)$-CIs for the selected parameters, and we prove that it controls the FCR at $\delta$ for confidence intervals that implicitly invert e-values; examples include those constructed via supermartingale methods, via universal inference, or via Chernoff-style bounds, among others.
The e-BY procedure is admissible, and recovers the BY procedure as a special case via a particular calibrator.
Our work also has implications for post-selection inference in sequential settings, since it applies at stopping times, to continuously-monitored confidence sequences, and under bandit sampling.
We demonstrate the efficacy of our procedure using numerical simulations and real A/B testing data from Twitter.
 \end{abstract}

\tableofcontents

\section{Introduction}
\label{sec:Intro}

One of the most classical problems in statistics is the problem of parameter estimation e.g.\ estimating the mean of a distribution. The time-tested solution is to produce a confidence interval (CI) in the parameter space that covers the true parameter with high probability. However, many scientists are not simply interested in estimating a single parameter --- they may have many \textit{potentially interesting} parameters to estimate concurrently.

For example, a scientist might be experimenting with $K$ vaccines for a strain of virus.
She may only be interested in reporting CIs for the vaccines that the data suggest to be effective. However, she does not know which vaccines may be effective until after she looks at the data. A reasonable thing she might do is to report the usual 95\% CIs for the vaccines where the CI is positive (i.e., the entire interval for the treatment effect is above zero). Consequently, she is using the \textit{same data} to both select the parameters she wishes to estimate, and construct estimates of these parameters. As a result, the reported (uncorrected) CIs do not provide valid statistical coverage. To see this, consider a scenario where none of the vaccines reduce mortality. Any reported CI will have a coverage probability of 0\%, since the CI for a vaccine is reported only if it is positive, meaning that it excludes the true effect of zero.

\chadded{The above example illustrates the issue of selection bias for post-selection inference, as prominently noted by \citet{benjamini_false_discovery_2005a}. Motivated by this example, let us quickly introduce the formal problem setup.

Let $\Pcal$ denote the universe of all possible distributions for the data.
Let $\vartheta : \mathcal{P} \to \Theta$ denote a functional (or parameter) that maps distributions to parameter values lying in some set $\Theta$. 
A scientist observes some data, \(\mathbf{X} = (X_1, \dots, X_K)\) that is drawn from an unknown distribution, $\pdist^*$, and we are potentially interested in the values of $K$ of its functionals\footnote{Technically these functionals could each lie in different sets $\Theta$ but this complicates notation, and we will anyway not explicitly need these sets later in the paper. Note that each $\theta_i$ need not be bijective.
For example, $\vartheta_i$ could capture the median of a distribution.} $\vartheta_1,\dots,\vartheta_K$, but our interest in them depends on the unknown parameter values \(\boldsymbol\theta^*  \coloneqq (\theta_1^*, \dots, \theta_K^*)\), where $\theta^*_i := \vartheta_i(\pdist^*)$.

Let $\mathbb{E}_\pdist$ and $\mathbb{P}_\pdist$ denote the expectation and probability under a distribution $\pdist$, respectively, although we sometimes suppress the subscript in the case of $\pdist^*$ for simplicity. 
For each $i \in [K]$, we assume that from $X_i$, the scientist can construct a marginal ($1 - \alpha$)-confidence interval for $\theta^*_i$, for any desired level $\alpha \in [0,1]$, which is a set $C_i(\alpha) \subseteq \Theta$ 
such that
\(
\mathbb{P}(\theta^*_i \notin C_i(\alpha)) \leq \alpha,
\)
or more explicitly, for any $\pdist \in \Pcal$, we would have
\(
\mathbb{P}_{\pdist}(\vartheta_i(\pdist) \notin C_i(\alpha)) \leq \alpha. 
\)
    
    The scientist uses the data $\mathbf{X}$ to select a subset of ``interesting'' parameters, \(\selset \subseteq [K]\), using some potentially complex data-dependent selection rule \(\selalg: \mathbf{X} \mapsto \selset \). The scientist must then devise confidence levels for the CI of each selected parameter, $\{\alpha_i\}_{i \in \selset}$, that \emph{can depend on the data }$\mathbf{X}$. The  \emph{false coverage proportion} (FCP) and \emph{false coverage rate} (FCR) of such a procedure are:
\begin{align}
    \FCP \coloneqq \frac{\sum_{i \in \selset}\ind{\theta_i^* \not\in C_i(\alpha_i)}}{|\selset| \vee 1}, \qquad \FCR \coloneqq \expect\left[\FCP\right],
\end{align} where $a \vee b = \max(a,b)$. Our goal is to design a method for choosing $\{\alpha_i\}_{i \in \selset}$ which guarantees $\FCR \leq \delta$ for a predefined level $\delta \in [0, 1]$ provided by the scientists in advance,
\emph{regardless of what the selection rule $\selalg$ is, and in particular even if the rule is unknown and we only observe the selected set $\selset = \selalg (\mathbf X)$.}} \Cref{fig:PostSelectionFlow} illustrates the setup of this post-selection inference problem.
\begin{figure}[h]
    \centering
    \includegraphics[trim={0 3cm 0 1cm},clip,width=0.7\textwidth]{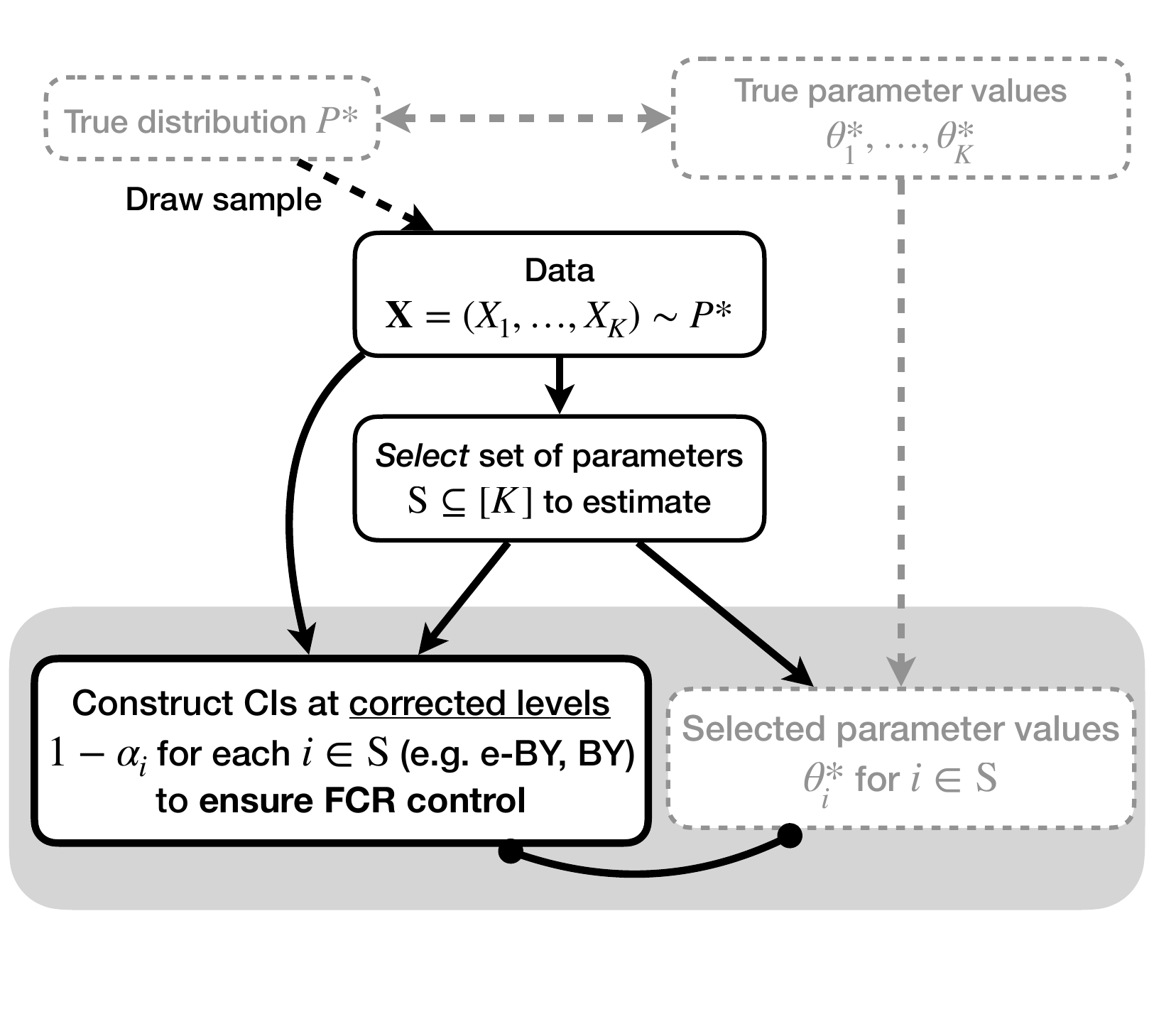}
    \caption{The post-selection inference with confidence intervals (CI) problem. We must choose corrected confidence levels \((1 - \alpha_i)\) for the marginal CI constructed for each \(\theta_i^*\) that is selected. The goal is to correct for the bias introduced by the selection rule and still provide a coverage guarantee for the resulting CIs. The coverage metric that we are interested in is the expected proportion of selected parameters that are not covered by their respective CIs, that is, the false coverage rate (FCR). Our corrected confidence level must guarantee that the FCR is below some fixed level of error \(\delta \in [0, 1]\).}
    \label{fig:PostSelectionFlow}
\end{figure}

We consider FCR over other error metrics for two main reasons. First, constructing a ``conditional CI'' which provides a coverage probability guarantee conditional on the selected set, $\selset$, requires knowing the selection \emph{rule} $\selalg$ beforehand --- while the scientist may sometimes be willing to provide this before seeing the data (and stick to it after seeing the data), this requirement still limits the usage of any such method because a scientist may wish to explore the data informally before deciding on a set $\selset$. Our e-BY method will overcome this limitation. Another option is to obtain a simultaneous coverage guarantee over all parameters, but this usually requires an overly conservative Bonferroni correction. We discuss these other metrics in more detail in \Cref{sec:CoverageTypes}. With that said, we will focus our attention on FCR control in this work.

\chadded{Our primary point of comparison is the BY procedure that was proposed in the same aforementioned paper by \citet{benjamini_false_discovery_2005a}. The BY procedure's choice of \(\{\alpha_i\}_{i \in \selset}\) and resulting guarantees depend upon  assumptions (or knowledge) of the dependence structure in $\mathbf{X}$ and the selection algorithm $\selalg$.
Under certain restrictions (omitted here for brevity) on $\selalg$, the BY procedure sets $\alpha_i  = \delta |\selset| / K$ to ensure that the FCR is controlled at level $\delta$ for mutually independent $X_1, \dots X_K$.
However, when no such assumptions can be made (i.e., under arbitrary dependence and an unknown selection rule) the BY procedure sets $\alpha_i = \delta |\selset| / (K\ell_K)$, where $\ell_K \coloneqq \sum_{i = 1}^K i^{-1} \approx \log K$ is the $K$th harmonic number. Clearly, the BY procedure produces much more conservative CIs when no assumptions can be made about dependence or selection.}

\chadded{In this paper, we introduce the \emph{e-BY procedure}, which achieves the best of both worlds: it requires no restrictions on the dependence structure or selection rule and produces CIs with larger error levels $\alpha_i = \delta |\selset| / K$  (yielding tighter CIs) and ensures $\FCR \leq \delta$ \emph{without any requirements on the dependence structure or selection rule.} However, its applicability is restricted to a smaller, but still quite broad, class of CIs. We refer to this class of intervals as ``e-confidence intervals'' (e-CIs) and we introduce them below.}

\paragraph{E-confidence interval (e-CI).} \chadded{
A crucial concept to defining an e-CI is the \textit{e-value} \citep{vovk_evalues_calibration_2020,shafer_testing_betting_2021,grunwald_safe_testing_2020,ramdas_admissible_2020}, a concept recently formulated in the hypothesis testing literature as an alternative to the classical p-value.
\begin{definition}
    A nonnegative random variable \(E\) is called an \emph{e-value} w.r.t.\ a set of distributions $\mathcal Q$, if \(\sup_{\pdist \in \mathcal{Q}} \expect_\pdist[E] \leq 1\).
\end{definition}
Above, $\mathcal Q$ represents the null hypothesis being tested and $E$ quantifies evidence against the null (large e-values are more evidence). As a simple example, we note that if $\mathcal{Q}=\{Q\}$ is a singleton, then the only admissible e-values are likelihood ratios of the form $dR/dQ$ for some alternative distribution $R$ \citep{ramdas_admissible_2020}.
}

\chadded{Given a particular functional $\vartheta$ of interest, let $\mathcal{P}_{\theta} := \{\pdist \in \mathcal{P}: \vartheta(\pdist) = \theta \}$ denote the set of all distributions with functional value $\theta$.
\begin{definition}
    Let $\{E(\theta)\}_{\theta \in \Theta}$ be a family of e-values such that for each $\theta \in \Theta$, $E(\theta)$ is an e-value w.r.t.\ $\mathcal{P}_{\theta}$.  
    For any \(\alpha \in [0, 1]\), we define the $(1-\alpha)$-\emph{e-confidence interval (e-CI)} 
as follows:
\begin{align}
    C(\alpha) = \left\{\theta \in \Theta: E(\theta) < \frac{1}{\alpha}\right\}.
    \label{eqn:ECI}
\end{align}
\end{definition}
\begin{proposition}
    Every e-CI is a valid CI.
\end{proposition}
\begin{proof}
   For any $\pdist \in \mathcal P_\theta$, we have \(\mathbb{P}_\pdist(\theta \not\in C(\alpha)) = \mathbb{P}_\pdist(E(\theta) \geq  1 / \alpha) \leq \alpha\), where the last inequality follows by Markov's inequality, since $\expect_\pdist[E(\theta)] \leq 1$ because $E(\theta)$ is an e-value with respect to any $P \in \mathcal P_\theta$.
\end{proof}}

\begin{figure}[h!]
    \centering
    \includegraphics[width=0.8\textwidth]{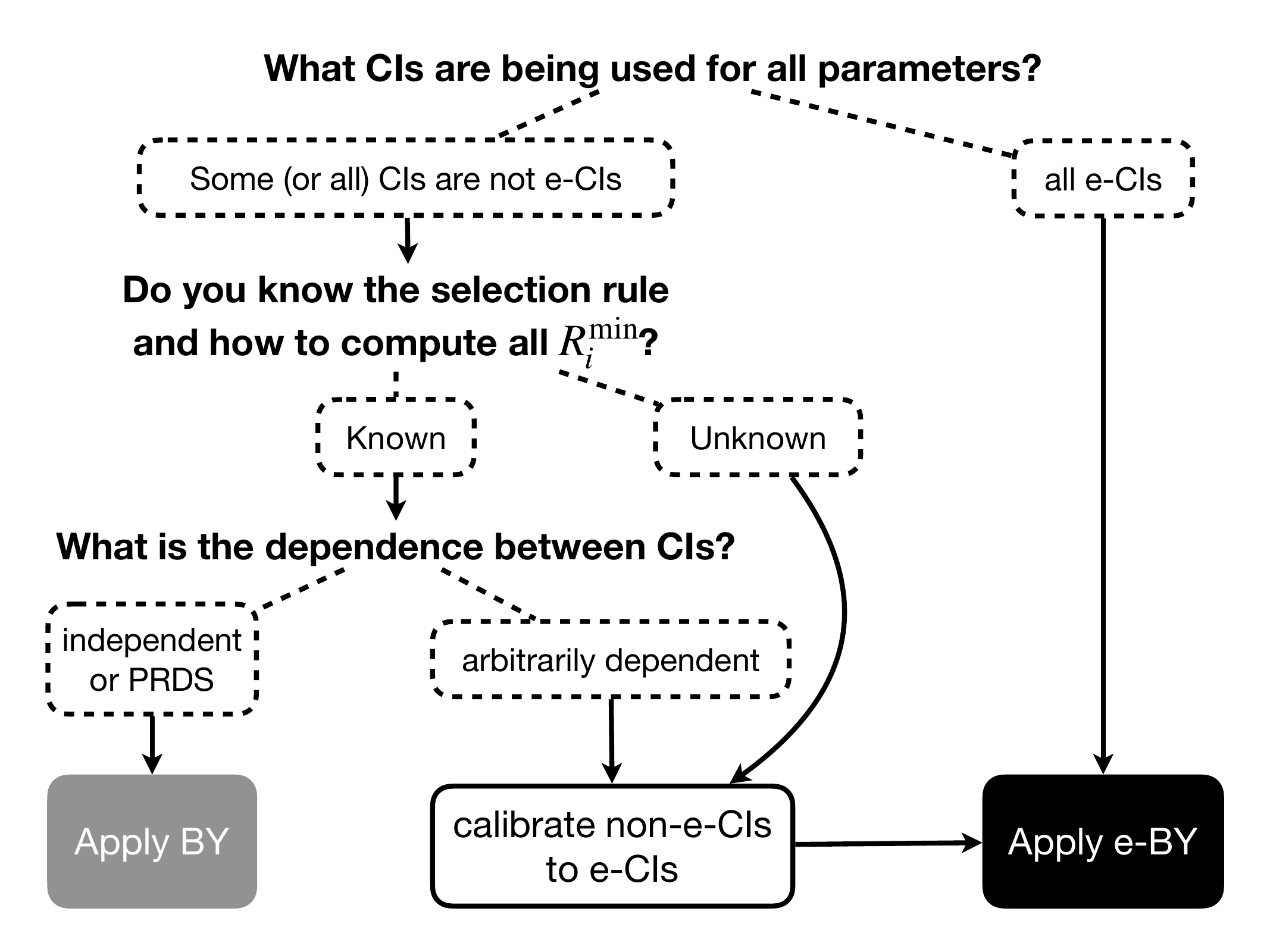}
    \caption{A flowchart for deciding when to use e-BY vs.\ BY. The only case where the BY procedure should be employed is when strong assumptions are made upon both the selection rule and the type of data dependence, while the e-BY procedure is a uniform improvement over the BY procedure in all other cases. If e-CIs can only be constructed for some parameters, and standard CIs for others, we can \textit{calibrate} standard CIs to e-CIs. In fact, calibration reduces the BY procedure to a special case of e-BY (\Cref{sec:Calibration}).}
    \label{fig:GeneralBY}
\end{figure}

\paragraph{Our contributions.} Our key contributions are twofold. First, we study the aforementioned novel ``e-CI'' subclass of CIs.
We show that a reasonably large array of existing CIs are already e-CIs. This new categorization sets the stage for our second main contribution: the e-BY procedure for FCR control (\Cref{def:eBY}). The e-BY procedure shows that one can attain tighter CIs than the BY procedure when no assumptions are made on the dependence structure or selection rule. The e-BY procedure is quite straightforward --- it simply sets the confidence level of the CI for each selected parameter to be \(1 - \delta|\selset|/K\), and this is sufficient to guarantee $\FCR \leq \delta$.

    On the other hand, the BY procedure \emph{requires knowledge of all counterfactual selections one would make on different data}. The BY procedure relies on quantities concerning these counterfactual selections (that cannot be derived without knowing the selection rule beforehand), and has to pay a costly correction factor (approximately logarithmic in the number of initial parameters, $K$) on the confidence levels when the selection rule is unknown.
    The BY procedure has such limits \emph{even when the data $X_1, \dots, X_K$ are mutually independent}.
    Thus, BY has no guarantees when parameters are selected by the scientist in an ad hoc fashion, which is often the case (e.g., exploratory data analysis, selection based on domain knowledge).
    In addition, the BY procedure --- even in optimal conditions of having independent data and a known selection rule --- cannot produce confidence levels that are tighter than those produced by e-BY. As a result, the e-BY procedure dominates the BY procedure when e-CIs are available or when one cannot make assumptions about either the selection rule or dependence structure. We elaborate on these points below:
\begin{enumerate}
\item [D1.]
\label{item:SelectionProcedure}
\textit{Selection rule must be known for the BY procedure to yield the tightest possible confidence intervals.} Under independence (and a more general condition of PRDS --- see \citet[Section 1.3]{benjamini_control_false_2001}), the BY procedure guarantees FCR control when the \(i\)th CI has a confidence level of \(1 -\delta |\selset|/ K\) only when the selection rule satisfies certain properties.
When the selection rule is not guaranteed to satisfy such properties, the BY procedure can only provide FCR control level $\delta$ by yielding CIs at a confidence level of \(1 - \delta|\selset| / (K\ell_K)\).
On the contrary, the e-BY procedure can provide the same level of FCR control by producing CIs at a confidence level of \(1 - \delta|\selset|/K\) i.e.\ the smallest level that can be achieved by the BY procedure even when the data is independent and the selection rule satisfies certain conditions.
\item[D2.]
\label{item:Dependence}
\textit{CIs produced by the BY procedure are larger when the CIs are arbitrarily dependent.} Even if the selection rule \(\selalg\) is known, the BY procedure produces more
conservative intervals than the e-BY procedure when the CIs for each parameter, $(C_1, \dots, C_K)$, can be arbitrarily dependent. The BY procedure must correct the confidence levels to be \(1 - \delta|\selset|/(K\ell_K)\) while the e-BY procedure maintains the same confidence levels of \(1 -\delta|\selset|/K\).
Arbitrarily dependent data can arise when we make multiple measurements of a single unit of experimentation, for example, measuring different gene expression levels in a single cell. Dependence is also prevalent in settings where data are collected sequentially --- we explore this usage further in \Cref{sec:ABTesting}, where we analyze both procedures using real data collected by an information technology company (Twitter) for the purposes of product testing.
    \item[D3.]
    \label{item:RequireECI}
    \textit{The e-BY procedure applies to e-CIs.} Of course, the e-BY procedure is beneficial solely when e-CIs can be constructed for the parameters of interest. Although ``e-CI'' is a new term, many existing CIs are e-CIs. These include non-asymptotic CIs based on Chernoff-type inequalities such as Hoeffding (\Cref{sec:BatchChernoff}), Bernstein, and empirical Bernstein CIs \citep{hoeffding_probability_inequalities_1963,bernsteinsergei_theory_probability_1927,maurer_empirical_bernstein_2009}, universal inference CIs \citep{wasserman_universal_inference_2020} (\Cref{sec:UI}), and CIs formed by stopping confidence sequences \citep{robbins_statistical_methods_1970,howard2020time,howard2021time,waudby-smith_estimating_means_2021,wang_catonistyle_confidence_2022} (\Cref{sec:Sequential}). Thus, finding powerful e-CIs is not a major limitation in many applications. Further, we will demonstrate in \Cref{sec:Calibration} that we can \textit{always} construct an e-CI from a CI, and one can use specific calibrators to recover the BY procedure (\Cref{sec:special-case}). Hence, if one has both CIs and e-CIs, calibrating the CIs to e-CIs and applying e-BY dominates the BY procedure.
\end{enumerate}

We summarize the tradeoffs between the e-BY and BY procedures in Table~\ref{tab:MethodTable}.

\begin{table}[h!]
    \centering
    \begin{tabular}{|c||c|c|c|}
        \hline
        \textbf{Procedure} & \textbf{Type of CI} & \textbf{Selection rule} & \textbf{Dependence}\\
        \hline
        { BY} & general & penalized for unknown \(\selalg\) (D1) & penalized (D2)\\
        \hline
        { e-BY} & e-CIs (D3) & no penalty & no penalty\\
        \hline
    \end{tabular}
    \caption{Tradeoffs between the e-BY procedure in this paper and the BY procedure \citep{benjamini_false_discovery_2005a}. The confidence levels of the e-BY procedure are not penalized for certain selection rules or for different types of dependencies between CIs. However, the e-BY procedure only being applicable to e-CIs, a special class of CIs derived from e-values \citep{vovk_evalues_calibration_2020,shafer_testing_betting_2021,grunwald_safe_testing_2020,ramdas_admissible_2020}.}
    \label{tab:MethodTable}
\end{table}

\paragraph{Outline} Since the e-BY procedure applies only to e-CIs, we describe several useful examples of e-CIs in \Cref{sec:ECI}. In \Cref{sec:eBY}, we prove that e-BY controls FCR under any selection rule and dependence structure, and contrast it with the BY procedure. To complement our validity results, we also show that the e-BY procedure is unimprovable in some notions --- \Cref{sec:SharpFCR} shows that the e-BY procedure has sharp control of the FCR and \Cref{sec:Admissible} shows that e-BY is admissible among a general class of e-CI reporting procedures. Lastly, we compare the empirical behavior of the e-BY and BY procedures on simulations in \Cref{sec:Simulations} and real user metric data from A/B testing experiments at Twitter in
\Cref{sec:ABTesting}.

\section{Confidence intervals obtained via e-values (e-CIs)}
\label{sec:ECI}

E-CIs are are referred to ``warranty sets'' by \citet{shafer_testing_betting_2021} in the context of game-theoretic statistics, and \citet{vovk2022confidence} defined a similar notion under the name of an \textit{e-confidence region}, although their version allows for thresholds smaller than 1.
Recent work has shown that e-CIs are pertinent to or arise naturally in many settings. For example, universal inference \citep{wasserman_universal_inference_2020} allows us to construct e-CIs for estimation of certain parameters in nonparametric and high-dimensional settings where the only requirement is having access to the likelihood functions for each null distribution. For sequential testing and estimation problems, e-processes, the sequential counterpart of e-values, play a central role in constructing time-uniform versions of CIs, known as confidence sequences \citep{robbins_statistical_methods_1970,howard2021time,waudby-smith_estimating_means_2021,wang_catonistyle_confidence_2022}. Further, many classical statistical tools already construct e-CIs, since they are implicitly using e-values or e-processes under the hood e.g.\ likelihood ratio tests and Chernoff methods~\cite{howard2020time}. Hence, limiting ourselves to e-CIs does not restrict the regime we may perform efficient inference in. In some cases, e-CIs might be the best tool we have for the job. For example, when estimating the mean of bounded random variables, the e-CIs in \citet{waudby-smith_estimating_means_2021} are empirically among the tightest known CIs for sequential estimation of the mean; see also \cite{orabona_tight_concentrations_2021}.
On the other hand, when we are in a setting where e-CIs are not the default choice of CI, one can \textit{calibrate} any CI into an e-CI. In fact, we can demonstrate that the BY procedure is a special case of the e-BY procedure through calibration. We elaborate on a couple of the aforementioned examples below.

\subsection{E-CIs from universal inference}
\label{sec:UI}
A particularly useful method for constructing e-CIs is through universal inference \citep{wasserman_universal_inference_2020}. Universal inference is a generalization of the standard likelihood ratio test that allows for testing against composite null hypotheses. Hence, e-CIs formed through universal inference only require access to likelihood functions (or a function that upper bounds the likelihood functions) for each possible distribution.

The simplest form of the universal inference e-CI can be derived from the split likelihood ratio test of \citet{wasserman_universal_inference_2020}. Assume that we are given \(A_1, \dots, A_{2n}\) i.i.d.\ samples. We first split this sample into two datasets \(D_0\) and \(D_1\) --- for simplicity we can assume that the two datasets are of equal sizes of \(n\) samples, although they do not have to be.  Let \(\pdist_1 \in \mathcal{P}\) be some distribution that we choose based solely on \(D_1\). In essence, \(\pdist_1\) is a guess of the ``most likely alternative'' and chosen to make the resulting e-value as large as possible. Consequently, the more accurately \(\pdist_1\) models the true distribution, the tighter the resulting e-CI will be. Let \(\mathcal{L}_0(\pdist)\) denote the likelihood of \(D_0\) under a distribution, \(\pdist\). We define the universal inference e-value as follows:
\begin{align}
    E_{\mathrm{UI}}(\theta) \coloneqq \frac{\mathcal{L}_0(\pdist_1)}{\argmax_{\pdist_0 \in \Pcal_\theta} \mathcal{L}_0(\pdist_0)}.
    \label{eqn:UIEvalue}
\end{align}

\begin{proposition}[Split universal inference e-CI \citep{wasserman_universal_inference_2020}]
For any parameter \(\theta \in \Theta\), \(E_{\mathrm{UI}}(\theta)\) is an e-value w.r.t.\ to \(\theta\) for any choice of procedure that derives \text{\normalfont \(\pdist_1\)} from \(D_1\). Consequently, the following set is a valid e-CI associated with \(E_{\mathrm{UI}}\):
\begin{align}
    C_{\mathrm{UI}}(\alpha) \coloneqq \left\{\theta: \frac{\mathcal{L}_0({\normalfont \pdist}_1)}{{\normalfont \argmax}_{\pdist_0 \in \Pcal_\theta} \mathcal{L}_0(\pdist_0)} < \frac{1}{\alpha}\right\}.
\end{align}
\end{proposition}

As a result, universal inference ensures that a nontrivial e-CI exists in any situation where the likelihood (or an upper bound on it) is known. In many settings, universal inference remains the only method for deriving nontrivial tests and CIs e.g.\ estimating the number of mixtures in a Gaussian mixture model when the dimension is greater than 1.

\subsection{E-CIs from stopped confidence sequences}
\label{sec:Sequential}

In the sequential setting, we assume that we receive samples of data, \(A_1, A_2, \dots\), in a stream. The goal is to produce a CI that is valid when the number of samples is data dependent, i.e., the user has the option to continuously monitor the data and continue or stop sampling based on the values observed so far. A typical sampling strategy is for a scientist to stop sampling as soon as she has evidence to reject the null hypothesis.  \citet{ramdas_admissible_2020} showed that any admissible sequence of CIs that is valid under adaptive sampling must be an e-CI.

Denote the sigma-algebra formed at each time step that contains all the random samples seen so far as \(\filtration_t \coloneqq \sigma(\{A_i\}_{i \leq t})\), with \((\filtration_t)_{t \in \naturals}\) being the corresponding filtration. A random variable \(\tau \in \naturals\) is a \textit{stopping time}  if \(\ind{\tau = t}\) is measurable w.r.t.\ \(\filtration_t\). This means that whether $\tau$ ``stops'' at time $t \in \naturals$ can only depend on $(A_1, \dots A_t)$. Since samples arrive one at a time, we can consider a sequence of intervals \((C^t(\alpha))\), where \(C^t(\alpha)\) is the CI we construct after collecting $t$ samples. However, a sequence of $(1-\alpha)$-CIs cannot guarantee that the CI at a stopping time, \(C^\tau(\alpha)\), is also a $(1 - \alpha)$-CI. In fact, \citet{howard2021time} provide simulations showing that the coverage probability can be drastically less than \((1 - \alpha)\). Thus, we want to strengthen the definition of a standard CI by providing a coverage guarantee at stopping times, as opposed to only fixed sample sizes.
\chadded{
\begin{definition}
    A $(1 - \alpha)$-\emph{confidence sequence} for a functional $\vartheta$ is a sequence of intervals $(C^t(\alpha))_{t \in \naturals}$ such that \(C^\tau(\alpha)\) is an \((1 - \alpha)\)-CI for any stopping time \(\tau\). It also has the following equivalent definition \citep[Lemma 2]{ramdas_admissible_2020}: for any $\pdist \in \Pcal$ and $\alpha \in [0, 1]$, we have
\begin{align}
    \mathbb{P}_{\pdist}(\vartheta(\pdist) \in C^t(\alpha) \text{ for all }t \in \naturals) \geq 1 - \alpha .
\end{align}
\end{definition}}

The formulation of a CS in the above definitions emphasizes the \emph{time-uniform} coverage a CS provides --- a CS ensures the probability $\theta^*$ is in the CS at every time step is high. One way to construct such an object is using a family of \emph{e-processes}, which are the sequential versions of the e-values.
\begin{definition}
    Let the index set \(\mathbb{I}\) be $\naturals$ or $(0,\infty)$. An \emph{e-process} w.r.t.\ to some filtration \((\filtration_t)_{t \in \mathbb{I}}\) and set of distributions \(\Pcal\) is defined as a sequence of random variables \((E^t)_{t \in \mathbb{I}}\) which are all nonnegative under any \(\pdist \in \Pcal\) and satisfy \(\sup_{\pdist \in \Pcal} \expect_{\pdist}[E^\tau] \leq 1\), i.e.,\ \(E^\tau\) is an e-value w.r.t.\ \(\Pcal\), for stopping time \(\tau\) w.r.t. \((\filtration_t)\).
\end{definition}

For recent literature relevant to e-processes, see \cite{grunwald_safe_testing_2020,ramdas_admissible_2020,howard2020time,howard2021time}.
E-processes are a superset of nonnegative supermartingales (one can see nonnegative supermartingales are e-processes as a result of the optional stopping theorem). Thus, a sequence of intervals constructed through applying the formula in \eqref{eqn:ECI} to a family of e-processes \((E^t(\theta))_{t \in \naturals}\) is a CS. For example, the following is a Hoeffding-esque e-process for testing the mean of a bounded random variable \citep{hoeffding_probability_inequalities_1963,waudby-smith_estimating_means_2021}:
\begin{align}
    E_{\mathrm{Hoef}}^t(\theta) \coloneqq \exp\left(\sum\limits_{i = 1}^t\lambda_i(A_i - \theta) - \frac{\lambda_i^2}{8}\right),
\end{align} where \((\lambda_t)\) is predictable w.r.t.\ \((\filtration_t)\) , i.e.,\ \(\lambda_t\) is measurable w.r.t.\ \(\filtration_{t - 1}\) for all \(t \in \naturals\).
We can define a two-sided e-CI based on \((E_{\mathrm{Hoef}}^t(\theta))\) as follows:
\begin{align}
    C_{\mathrm{Hoef}}^t(\alpha) \coloneqq \left(\frac{\sum\limits_{i = 1}^t \lambda_i A_i}{\sum\limits_{i = 1}^t \lambda_i} \pm \frac{\log(2 / \alpha) + \frac{1}{8}\sum\limits_{i = 1}^t \lambda_i^2}{\sum\limits_{i = 1}^t \lambda_i}\right).
\end{align} Notice that setting \(\lambda_t = \sqrt{\log(2 / \alpha) / n}\) recovers the classic Hoeffding CI \citep{hoeffding_probability_inequalities_1963} for the specific sample size of \(n \in \naturals\) and fixed $\alpha \in [0, 1]$. We discuss how to optimize these parameters in the post-selection inference setting where there are  multiple potential values of $\alpha$ in \Cref{sec:BatchChernoff}.

Another implication of the above formulation is that the e-process view allows one to extend the coverage guarantee of existing CIs to be valid sequentially.
This coverage improvement is not only applicable to Hoeffding's, but also a wide range of Chernoff type inequalities \citep{howard2020time,howard2021time,wang_catonistyle_confidence_2022}.  Similarly, the universal inference e-CI formulated in \eqref{eqn:UIEvalue} can also be extended to a sequential form, where \(\pdist_1\) may be recomputed at each time step based on the previously observed samples. For bounded random variables, one can construct e-processes that directly correspond to the wealth process of a sequence of fair gambles \citep{waudby-smith_estimating_means_2021}.
Since e-CIs are necessary for sequential inference, the e-BY procedure is a strict improvement over the BY procedure in settings that allow for adaptive sampling and stopping such as in multi-armed bandits and the A/B testing setting we discuss in \Cref{sec:ABTesting}.

\subsection{E-CIs constructed from CIs through calibration}
\label{sec:Calibration}

Although the definition of an e-CI is much more specific than a general CI, we will present a method for calibrating arbitrary marginal CIs to e-CIs using a method for p-value to e-value calibration \citep{vovk_logic_probability_1993b, sellke_calibration_values_2001a,shafer_test_martingales_2011,shafer_game_theoretic_2019,vovk_evalues_calibration_2020}.

\begin{definition}
A \emph{calibrator} is a nonincreasing, upper semicontinuous function \(f: [0, 1] \rightarrow [0, \infty]\) where \(\int_{0}^1 f(x)\ dx \leq 1\).
\label{def:Calibrator}
\end{definition}

Define \(f^{-1}\) to be the dual of the calibrator \(f\) , i.e.,\ \(f^{-1}(x) \coloneqq \sup \{p: f(p) \geq x\}\). When \(f\) is invertible, \(f^{-1}\) is the inverse of \(f\). Using \(f^{-1}\), we can convert any CI to an e-CI. Also we define two properties about any set-valued CI function \(C: [0, 1] \mapsto  2^{\Theta}\). Let $C$ be \textit{nonincreasing} if, for any $\alpha, \beta \in [0, 1]$, $\alpha \leq \beta$ implies that $C(\alpha)\supseteq C(\beta)$. Further, define \textit{continuous from below} to be the property that for any $\alpha \in [0, 1]$, $C(\alpha) = \bigcup\limits_{\beta > \alpha} C(\beta)$.
\begin{theorem}
Let \(C: [0, 1] \mapsto  2^{\Theta}\) be a nonincreasing function that is continuous from below such that \(C(\alpha)\) produces a \((1 - \alpha)\)-CI, and \(f\) be a calibrator (with dual calibrator \(f^{-1}\)). Then, the following set is a \((1 - \alpha)\)-e-CI:
\begin{align}
    C^{\mathrm{cal}}(\alpha) &= C\left(f^{-1}\left(\frac{1}{\alpha}\right)\right).
\end{align}\label{thm:Calibration}
\end{theorem}
We provide a proof of \Cref{thm:Calibration} in \Cref{sec:CalibrationProof}. As a consequence of this theorem, we can construct a nontrivial e-CI from any nontrivial CI and a dual calibrator. Of course, this prompts the question: when should one use e-CIs with the e-BY procedure instead of CIs with the BY procedure?
Generally speaking, the calibrated e-CI will be looser than the original CI, since \(f(x) = 1/x\) is not a calibrator. However, we show in \Cref{sec:eBY} that applying e-BY to calibrated e-CIs is actually equivalent to applying the BY procedure to the original CIs when the selection rule is unknown or the original CIs are dependent. Thus, calibrating CIs to e-CIs creates e-CIs across all parameters, in the case where some parameters initially do not have e-CIs.
\section{The e-BY procedure}
\label{sec:eBY}

Now, we formally define the e-BY procedure as follows.

\begin{definition}
    The \emph{e-BY procedure} at level \(\delta \in [0, 1]\) sets \(\alpha_i = \delta |\selset| / K\) for each \(i \in \selset\).
    \label{def:eBY}
\end{definition}
We show that a FCR bound can be proven quite simply given the fact that e-CIs are constructed for each selected parameter.
\begin{theorem}
    Let \(C_i(\alpha)\) be a \((1 - \alpha)\)-e-CI for each \(i \in [K]\) and $\alpha \in [0, 1]$. Then, the e-BY procedure in \Cref{def:eBY} ensures \(\FCR \leq \delta\) for any \(\delta \in (0, 1)\) under any dependence structure between \(X_1, \dots, X_K\), and any selection rule \(\selalg\).
\label{thm:eBY}
\end{theorem}
\begin{proof}
We directly show an upper bound for the FCR as follows:
\begin{align}
    \text{FCR} &= \mathbb{E}\left[ \frac{\sum_{i \in \selset} \ind{\theta_i^* \notin C_i(\delta |\selset|/K)} }{|\selset| \vee 1} \right]\\
               &= \mathbb{E}\left[ \frac{\sum_{i \in [K]} \ind{E_i(\theta_i^*)|\selset|\delta / K > 1} \cdot \ind{i \in S} }{|\selset| \vee 1} \right]\\
    &\leq \sum_{i \in [K]} \mathbb{E}\left[ \frac{E_i(\theta_i^*) |\selset| \delta} {K(|\selset| \vee 1)} \right] =  \sum_{i \in [K]} \frac{\delta}{K} \mathbb{E}\bigg[E_i(\theta_i^*)  \cdot \frac{ |\selset|}{|\selset| \vee 1} \bigg] \leq \delta,
\end{align} where \(\ind{\cdot}\) is the indicator function. The first inequality is because $\ind{x>1}\le x$ for all $x\ge 0$. The second inequality is a result of the definition of the e-value for $\theta^*_i$ having its expectation under $\pdist^*$ be upper bounded by 1. This achieves our desired bound.
\end{proof}

\begin{remark}
    FCR control of the e-BY procedure implies FDR control of the e-BH procedure \citep{wang_false_2020}, while the converse is not true. The arbitrary nature of $\selset$ is unique to the post-selection inference problem (as e-BH can only reject ``self-consistent'' \citep{blanchard_two_simple_2008} sets of hypotheses), and defining the concept of e-CIs is key to achieving FCR control. We expand on the relationship between these procedures in \Cref{sec:MultipleTesting} and visualize their relationships in \Cref{fig:MultipleTestingFlow}. We also introduce a directional variant of e-BH that uses the FCR control of the e-BY procedure to achieve directional FDR control in \Cref{sec:FSR}.
\end{remark}

\paragraph{False coverage control for data-dependent \(\delta\).} An interesting result of this proof is that we retain a form of control on the FCP even if \(\delta\) is chosen data dependently. Define \(\FCP(\delta')\) to be the FCP of a post-selection inference procedure at level \(\delta'\). \begin{corollary}\label{corollary:supfcr}
The following is always true for the e-BY procedure (regardless of selection rule or dependence in \(\mathbf{X}\)):
\begin{align}
    \expect\left[\sup_{\delta' \in (0, 1]} \frac{\FCP(\delta')}{\delta'}\right] \leq 1. \label{eqn:SupremumBound}
\end{align} Here, \(\FCP(\delta')\) is FCP of the e-BY procedure at level \(\delta'\) , i.e.,\ \(\alpha_i = \delta' |\selset| / K\) for all \(i \in \selset\).
\end{corollary}
Satisfying \eqref{eqn:SupremumBound} is a stronger condition than FCR control, since \eqref{eqn:SupremumBound} implies FCR control for any fixed \(\delta \in [0, 1]\).
 This error guarantee follows from an argument similar to the proof of \Cref{thm:eBY}.
\chadded{When testing $K$ hypothesis to control the false discovery rate, we note that a similar result for the false discovery proportion (FDP) of the e-BH procedure is implied by this corollary: 
\begin{align}
    \expect\left[\sup_{\delta' \in (0, 1]} \frac{\mathrm{FDP}(\delta')}{\delta'}\right] \leq 1. \label{eqn:ebh-fdp}
\end{align}
This can also be observed directly from the proof of e-BH and we include the details in \Cref{sec:EBHtoEBY} for completeness. When specialized to a single hypothesis test, the latter observation amounts to saying that for any e-value $E$, we have
\[
\expect\left[\sup_{\delta' \in (0, 1]} \frac{\ind{E \geq 1/\delta'}}{\delta'}\right] \leq 1,
\]
which is captured in~\cite[Lemma 1]{wang_false_2020}.
One can view the post-hoc type-I safety discussed by \citet{grunwald_neyman-pearson_2023} as a rephrasing of this latter result.}

\subsection{Comparison with the BY procedure}

The FCR guarantee for e-BY in \Cref{thm:eBY} solely relies on the definition of an e-CI. Since the expectation property of e-values holds regardless of the type of dependence, proving FCR control of e-BY does not require fine grained analysis of how the distribution of an e-value changes when conditioned on other e-values. On the other hand, the BY procedure does require more detailed analysis because its proof only uses the marginal coverage property of each CI. Hence, the BY algorithm requires different levels of corrections based on the dependence structure and selection rule.
\begin{definition}
    The \emph{BY procedure} at level \(\delta \in [0, 1]\) sets \(\alpha_i = \delta R_i^{\min} / K\) if \(X_1, \dots X_K\) are independent (or PRDS \citep{benjamini_false_discovery_2005a}), \(R_i^{\min}\), as formulated in \eqref{eqn:Rmin}, is known, and \(R_i^{\min} > |\selset| / \ell_K\). Otherwise, it sets \(\alpha_i = \delta |\selset| / (K\ell_K)\).
\end{definition}
\begin{fact}[Theorems 1 and 4 from BY \citep{benjamini_false_discovery_2005a}]
The BY procedure ensures that \(\FCR \leq \delta\). \label{fact:BY}
\end{fact}
\vspace{-10pt}
A key difference in the tightness of CIs between the e-BY and BY procedures is how the selection rule affects the confidence levels chosen for the CIs. In the e-BY procedure, the confidence levels are proportional to \(|\selset|\), the number of selected parameters. On the other hand, the BY procedure requires knowledge of \(R_i^{\min}\), which is formulated as follows for each \(i \in \selset\):
\begin{align}
R_{i}^{\min} \coloneqq \min\ \{&|\selalg(\mathbf{X}_{i\rightarrow x_i})|: x_i \in \samplespace_i, ~i \in \selalg(\mathbf{X}_{i\rightarrow x_i})\},
\label{eqn:Rmin}
\end{align} where $\mathbf{X}_{i\rightarrow x_i} \in \samplespace$ is the data $\mathbf{X}$ with the $i$th component set to $x_i$.
\paragraph{Computing \(R_i^{\min}\) requires knowledge and analysis of the selection rule} Clearly, \(R^{\min}_i\) is a quantity that has to be calculated from the selection rule, \(\selalg\). In practice, however, \(R^{\min}_i\) may be impossible to calculate. The selection rule may not be fully specified or the counterfactual behavior of the selection rule may be difficult to analyze.
Post-selection inference is often used for \textit{exploratory data analysis} in the sense described by \citet{goeman_multiple_testing_2011}.
Hence, the selection rule is often \textit{not} known by the scientist ahead of time --- she is developing her selection plan on the fly as she explores the data!
Another common scenario is that scientist carrying out the experiment, performing analysis, and ultimately producing the CIs may not be the one deciding upon the selection to report e.g.\ the scientist's manager makes the decision. In either case, the selection rule is not known.
As an aside, the e-BY procedure can also be used by the manager in the prior example to produce CIs herself, since she does not need to know the experimental design and dependencies in the data to construct CIs using the e-BY procedure.
Further, even if the selection rule was known beforehand, it still might be difficult to analyze \(R_i^{\min}\). For example, the selection procedure could be iterative: at each iteration, the scientist considers a selection set, observes some property of the CIs produced by the e-BY procedure for the candidate selection set and then selects a new candidate selection set based on her observations. Eventually, the scientist may finalize the selection set based on some sort of convergence criteria. In this selection scheme, the full procedure is not specified, and quite difficult to specify since it requires knowing what the scientist would do at every iteration if experimental data were different. Thus, it is often impractical to assume \(R_i^{\min}\) is known.

The other main advantage of the e-BY procedure can be seen when the dependence between CIs is arbitrary --- there is a clear gap between the confidence levels for the BY procedure and the e-BY procedure. The BY procedure requires an extra correction of \(\ell_K \approx \log K\) on the error probability under arbitrary dependence, while the e-BY procedure behaves the same way under any form of dependence. Thus, if the CIs are e-CIs, the e-BY procedure yields tighter CIs than the BY procedure when the data is arbitrarily dependent or the selection rule is unknown or has \(R_i^{\min} \leq |\selset| / \ell_K\).
\subsection{BY is special case of e-BY through e-CI calibration}
\label{sec:special-case}
By using a specific calibrator with the method for calibrating CIs into e-CIs discussed in \Cref{sec:Calibration}, we can show that the BY procedure (with no assumptions on dependence or selection rule) is a special case of the e-BY procedure. Define the following calibrator when there are \(K\) hypotheses and the desired FCR control is \(\delta\):
\begin{align}
    f^{\rmBY(\alpha, K)}(x) \coloneqq \frac{K}{\alpha}
    \left(\left\lceil \frac{x}{\alpha/ (K\ell_K)} \right\rceil \vee 1 \right)^{-1}\ind{x \leq \alpha / \ell_K}.
\end{align}
\begin{proposition}
\(f^{\rmBY(\delta, K)}\) is an upper semicontinuous calibrator for all \(\delta \in (0, 1)\) and \(K \in \naturals\).
\end{proposition}
\begin{proof}
Clearly, \(f^{\rmBY(\delta, K)}\) is nonincreasing, and we note that:
\begin{align}
    \int_0^1 f^{\rmBY(\delta, K)}(x)\ dx &=\sum\limits_{j = 1}^K \frac{\delta}{K \ell_K} \cdot \frac{K}{j\delta} = \frac{1}{\ell_K} \sum\limits_{j = 1}^K \frac{1}{j} = 1.
\end{align} Hence, we proved that \(f^{\rmBY(\delta, K)}\) satisfies the properties to be a calibrator.
\end{proof}

By our result in \Cref{thm:Calibration}, we can now show that the e-BY procedure with \(f^{\rmBY(\delta, K)}\) calibrated e-CIs is identical to the BY procedure under arbitrary dependence.
\begin{corollary}
For any arbitrary CI constructors, denoted as \(C_i: [0, 1] \mapsto 2^{\Theta_i}\) for each \(i \in [K]\), that are continuous from below. The BY procedure produces the same CIs as the e-BY procedure with e-CI constructors \(C_i^{\mathrm{cal}}\) that were calibrated through  \(f^{\rmBY(\delta, K)}\) for each \(i \in [K]\).
\end{corollary} Note that one take the limit from below at every $\alpha \in [0, 1]$ of a nonincreasing CI constructor $C$ to create a CI constructor $C'$ that is continuous from below. $C'$ has valid coverage and is no wider than the original $C$ for every $\alpha$, so the continuity from below is only a technical requirement.

The above corollary follows from the fact that \Cref{thm:Calibration} and the choice of \(f^{\rmBY(\delta, K)}\) as calibrator implies the following:
\begin{align}
    C_i\left(\frac{\delta j}{K\ell_K}\right) = C_i^{\mathrm{cal}}\left({f^{\rmBY(\delta, K)}}^{-1}\left(\frac{K \ell_K}{\delta j}\right)\right) = C_i^{\mathrm{cal}}\left(\frac{\delta j}{K}\right)  \text{ for all }j \in [K].
\end{align}

Thus, the output of the BY procedure is equivalent to the output of the e-BY procedure with calibrated e-CIs for all possible sizes of the selection set \(\selset\). As a result, in a situation where there may be a mix of e-CIs and regular CIs that are arbitrarily dependent and the regular CIs, one can always calibrate the regular CIs to e-CIs using \(f^{\rmBY(\delta, K)}\) and then apply the e-BY procedure to produce CIs that are as tight as directly applying the BY procedure.

\chadded{\begin{remark}

We have not compared other calibrators, mostly because there is no uniformly best choice of calibrator to use: different ones will perform better in different settings. We recommend using the BY calibrator by default. The BY calibrator is a natural choice since it allows the resulting output of the e-BY procedure on calibrated CI to be no worse than the BY procedure (under dependence or unknown selection rule). It is also exact, in the sense for the resulting e-CI, the e-value corresponding to the true value of the parameter, $E(\theta^*)$, will have an expectation exactly equal to 1 if the original CI has exact coverage.

We have not defined admissibility for CI calibrators, but we conjecture that for some suitable definition, the BY calibrator we introduce is an admissible CI calibrator of CIs into e-CIs (since the corresponding p-to-e value calibrator is admissible). Furthermore, it is well suited to use with the e-BY procedure, since e-BY only produces CIs with error levels at multiples of $\delta / K$ --- we also conjecture  that it  is admissible among CI calibrators w.r.t.\ to the tightness of CIs produced by e-BY. We think it is likely that the BY calibrator is the only calibrator that allows one to produce CIs that were never worse than BY, making it a good default choice.
\end{remark}}

\section{FCR control of e-BY is sharp}
\label{sec:SharpFCR}

To complement our upper bound on the FCR, we demonstrate the FCR bound is indeed sharp for the e-BY procedure by formulating a setting where the FCR is arbitrarily close to \(\delta\). First, let \(\reals^-\) and \(\reals^+_0\) denote the negative and nonnegative reals, respectively. Now, let
$$
Z^t_i(\theta) =1+ W_i^t  -  \theta  t
~\text{for each }\theta \in \reals^-, ~t \in \reals^+_0, \text{ and }i \in [K],
$$
where $(W_i^t)_{t \in \reals^+_0}$ is a Brownian motion with drift $\theta_i^*$ (unknown to the scientist) and volatility $1$. Define $(E_i^t(\theta)  )_{t \in \reals^+_0}$ as the process obtained by stopping $(Z_i^t)_t$ when it hits $0$:
$$E^t_i (\theta) = Z^t_i(\theta)\ \ind{\inf_{s\in [0,t]} Z^s_i(\theta)>0}.$$

Assume that $(W_{i}^t)$ are independent across $i\in [K]$. Clearly, $(E_i^t(\theta_i^*))$ is a nonnegative martingale (and an e-process) for $\theta_i^*$, since a Brownian motion is a martingale, and stopped martingales are martingales as well.
\chadded{We imagine the scientist to be interested in estimating parameters whose true value $\theta^*_i$ is positive, and she uses the following experimentation setup and selection rule to do so. Fix $\gamma\in [1,1/\delta)$ as the threshold for the processes $(E_i^t(0))$ to be seen as ``interesting'' (and chosen for further experimentation) but not yet conclusive.}

We consider a two-step procedure. In the first step, we stop the $i$-th experiment at time
$$\tau_i \coloneqq \inf\{t\ge 0: W_i^t \ge \gamma - 1 \mbox{ or }  W_i^t   \le  -1\},$$
and we select $S$  such that
$i \in S$ if and only if $W_i^{\tau_i}(0) \geq \gamma - 1$ (i.e., $E_i^{\tau_i}(0) \geq \gamma$).
That is, we select $i$ such that $W_i^t + 1 = E_i^t(0) $ reaches $ \gamma$ (interesting) before it reaches $0$ (discarded). In other words, the set $\selset$ contains all interesting experiments in the first screening.
If $\gamma=1$ then we simply select all experiments.
Let $\beta = K/(\delta |\selset|) >\gamma$. Note that  $\beta $ is a function of $\selset$.
(Here the index $t$ of $W_i^t$ does not necessarily represent time, and it does not have to  synchronize across $i\in [K]$. The scientist can finish all step 1 experiments before moving on with the selected ones for step 2.)

In the second step, define another stopping time for the $i$-th experiment:
$$\eta_i=\inf\{t\ge 0: W_i^t\ge \beta - 1  \mbox{ or }  W_i^t   \le -1\}.$$
That is, we will stop the experiment if either $E_i^t(0) $ reaches a very high level $\beta$ or it is discarded.
Note that $\eta_i>\tau_i$ for $i\in\selset$ and $\eta_i=\tau_i $ for $i\not\in\selset$, meaning that we only continue those experiments that were deemed as interesting in the first step. As a result, $\{E_i^{\eta_i}(\theta)\}_{\theta \in \Theta}$ forms a family of e-values for each $i\in\selset$, and we can define the following one-sided e-CIs:
\begin{align}
    C_i(\alpha) &\coloneqq \left\{\theta \in \reals: E_i^{\eta_i}(\theta) < \frac{1}{\alpha}\right\}\\
    &= \left\{\theta \in \reals: 1 + W_i^{\eta_i} - \theta \eta_i < \frac{1}{\alpha} \text{ or } \inf_{s \in [0, \eta_i]} 1 + W_i^s - \theta s \leq 0\right\}\\
    &  =\left(\min\left(\frac{1 + W_i^{\eta_i} - \frac{1}{\alpha}}{\eta_i}, \inf_{s\in [0, \eta_i]} \frac{1 + W_i^{s}}{s}\right), \infty \right).
    \label{eqn:BrownianCI}
\end{align}

\begin{theorem} Let e-CIs be specified by \eqref{eqn:BrownianCI}. Then, \(\lim\limits_{\vecb{\theta}^*\uparrow \vecb{0}} \FCR = \delta\) , i.e.,\ limit of FCR for the e-BY procedure approaches \(\delta\) as \(\vecb{\theta}^*\) approaches \(\vecb{0}\) (all zero vector) from below.
\label{thm:Sharpness}
\end{theorem}
A full proof of this theorem can be found in \Cref{sec:SharpnessProof}.
\begin{remark}
    The other source of looseness of the FCR bound in \Cref{thm:eBY} is in the relationship between selection event of each parameter and the e-value of the true parameter, i.e., \(E_i(\theta^*_i)\). \chadded{In this example, we can see that the selection rule and the true value of $\mathbf{\theta}^*$ are conflicting in some sense. The scientist wants to select parameters with positive values, so when the parameters are actually all negative values, the CI for a selected parameter will probably not cover the corresponding true value $\theta_i^*$.} We discuss this further in \Cref{sec:Improvement} and observe some small improvements we may make on the e-BY procedure under stronger assumptions about the e-CIs.
\end{remark}
\chadded{\begin{remark}
    While the Brownian motion data generating process in our setup is stylized, it does reflect some realistic aspects of sequential data collection. The early stopping rule that halts as soon as an e-CI shows significance, and the selection of promising experiments for further testing, are both natural strategies for sequential experimentation (e.g., two stage designs, follow up studies). Hence, depending on the relationship between selection rule and the true value of the parameters (as illustrated in our example), it is plausible that the true FCR could be close to the FCR bound.

    On the other hand, the design of the parameters in this example is highly adversarial: $\vecb{\theta}^*$ needs to be very close to 0 from below to approach this upper FCR bound. 
    Therefore, in practice, we do not always expect the FCR of e-BY to be close to its upper bound, but the bound is still unimprovable in general.
    \end{remark}}
\section{Admissibility of the e-BY procedure}
\label{sec:Admissible}
Not only does the e-BY procedure provide sharp error guarantees, it is also admissible w.r.t.\ a general class of e-CI post-selection inference methods. \chadded{We will formally define the universe of \emph{CI reporting procedures} that e-BY belongs to, i.e., procedures which report a CI for the true value of each selected parameter. We will show that, among CI reporting procedures that control the FCR when given e-CIs, there exists no other procedure that produces strictly tighter CIs on selected parameters than e-BY.  To prove this, we define a notion of dominance for FCR controlling CI reporting procedures, and prove that e-BY is indeed admissible among this set of procedures.

Our notion of admissibility essentially requires that there is no FCR controlling procedure that produces narrower CIs on all the selected parameters uniformly over all possible inputs of e-CIs (by narrower, we mean at least as narrow on all instances, but strictly narrower on at least one instance).
To develop our formalisms, we will first describe a notion of an e-CI constructor, i.e., a function that maps an error level $\alpha$ to its corresponding $(1 - \alpha)$-e-CI.}

\subsection{Key properties of all e-CIs}
Recall that the $(1 - \alpha)$-e-CI \textit{associated} with a family of e-values $\{E(\theta)\}_{\theta \in \Theta}$ is given by
\begin{equation}\label{eq:associated}
    C(\alpha) = \left\{\theta \in \Theta : E(\theta) < \frac{1}{\alpha}\right\},~ \text{ for each } \alpha \in [0,1],
\end{equation}
where the convention $1/0=\infty$ is used. Note that it is important to formulate the e-CI by using the strict inequality $E(\theta) < 1 / \alpha$
instead of $E(\theta) \le  1 / \alpha$.
Otherwise all admissibility arguments fail, since one could always improve from $ \{\theta \in \Theta : E(\theta)  \le1 / \alpha\} $ to $ \{\theta \in \Theta : E(\theta) < 1 / \alpha\} $, as smaller CIs are statistically stronger.
\chadded{We define an e-CI constructor as the collection of e-CIs for each error level $\alpha \in [0, 1]$.
\begin{definition} Let $\Erm \coloneqq \{E(\theta)\}_{\theta \in \Theta}$ be a family of e-values for a parameter that takes values in $\Theta$. The corresponding \emph{e-CI constructor} is the function $C^{\Erm}: [0, 1] \rightarrow 2^\Theta$ which is defined as:
\begin{align}
    C^\Erm(\alpha) = \left\{\theta \in \Theta : E(\theta) < \frac{1}{\alpha}\right\},~ \text{ for each } \alpha \in [0,1].
\end{align}
\end{definition}

It is important to note that  $C^\Erm$ we define above is a \emph{random function}. We have not separated designations for random variables or functions and their realizations in this paper before, but we emphasize their distinction in this section. This distinction allows us to define the domain of realizations of e-CI, which in turn, is used to the notion of a CI reporting procedure.

Denote by the set $\ECI(\Theta)$ the set of possible realizations of e-CI constructors for the parameter space $\Theta$, that is, $C \in \ECI(\Theta)$ iff  $C$ is a possible realization of an e-CI constructor $C^\Erm$ for some family of e-values $\mathrm{E} = \{E(\theta)\}_\theta$ via \eqref{eq:associated} under some probability $\pdist \in \Pcal$.} We abuse vocabulary in this section and refer to the e-CI constructor, $C$, as an ``e-CI'' as well.
$\mathrm{ECI}(\Theta)$ captures the domain of allowed input of realized e-CIs, similar to  $[0,1]^K$ for input p-values and $[0,\infty]^K$ for input e-values in multiple testing.
It is important to formally describe $\mathrm{ECI}(\Theta)$ as it will be used as part of the domain of CI reporting procedures that we discuss later.

\begin{lemma}\label{lemma:ECINonincreasing}
For a set function $C:[0,1]\mapsto 2^\Theta$,
$C\in \mathrm{ECI}(\Theta)$ if and only if $C$ is nonincreasing and continuous from below.
\end{lemma}

We provide the proof in \Cref{sec:ProofECINonincreasing}. We see from Lemma \ref{lemma:ECINonincreasing}  that $\mathrm{ECI}(\Theta)$ is precisely the set functions from $[0,1]$ to $2^\Theta$ which are decreasing and continuous from below; note that these properties are common for any CI. In other words,
for a given CI  which may be computed with other methods than e-values,
we cannot exclude the possibility that it is an e-CI just by looking at the CI.

 \begin{lemma}\label{lemma:ECIEvalue}
For a given $C\in \mathrm{ECI}(\Theta)$,
it is realized by the family of e-values $\{E(\theta)\}_\theta$ if and only if $E(\theta)$ takes the value
$
E(\theta)=t(\theta)
$ whenever $E(\theta)\ge 1$ for all $\theta \in \Theta$,
where $t$ is given in \eqref{eq:deft}.
 \end{lemma}
This is a direct consequence of \eqref{eq:associated} and \eqref{eq:equi-t}.
Due to Lemma \ref{lemma:ECIEvalue}, for any family of e-values for $\Theta$ and  $\theta\in \Theta$, the probability of realizing $C\in \mathrm{ECI}(\Theta)$ is at most $1/t(\theta)$ for  $t$ given in \eqref{eq:deft}.
\subsection{CI reporting procedures}
\label{sec:CIreport}

Next, fix  $K$ parameter spaces $\mathbf{\Theta} \coloneqq( \Theta_1,\dots,  \Theta_K)$ of interest.
$\mathrm{ECI}_K$ denotes the set of all $K$-tuples of realizations of e-CIs, $\mathbf C=(C_i)_{i\in [K]}$, for the true values of these parameters (we omit the parameter spaces in $\mathrm{ECI}_K$ since they are fixed from now on). \chadded{Define $\mathbf{2}^{\Theta} \coloneqq \prod_{i = 1}^K 2^{\Theta_i}$ as the product space of the power set of each parameter's range. Using Lemma \ref{lemma:ECINonincreasing},  $\mathrm{ECI}_K$ is the collection of  $\mathbf C:[0,1]\to \mathbf{2}^{\Theta}$
whose components are nonincreasing and continuous from below. Now, we define a notion of a CI reporting procedure.}
\chadded{{
\begin{definition}
    \chadded{A \emph{CI reporting procedure} $\Phi: [K] \times ([0, 1]\rightarrow \mathbf{2}^{\Theta})  \rightarrow \mathbf{2}^{\Theta}$ takes a selected set of indices $S \subseteq [K]$ and the collection of e-CIs, $\mathbf C =(C_1,\dots,C_K)  \in\mathrm{ECI}_K$, as input  and it outputs a vector of intervals, $\Phi(S,\mathbf C) \in \mathbf{2}^\Theta$, satisfying}
 \begin{enumerate}[label=(\alph*)]
     \item \label{item:EmptyUnselected} \chadded{$\Phi(S, \mathbf C)_i = \emptyset$ for all $i \not \in S$.}
\item  \label{item:SelectionOnly} \chadded{$\Phi(S,\mathbf C)   =  \Phi(S,\mathbf C')
 $
if $\mathbf C_S=\mathbf C'_S$
where  $\mathbf C_S= (  C_i)_{i \in S} $; that is,   for a selected set of indices $S$, if two input vectors of CIs are identical for  indices in $S$, then  $\Phi$  does not distinguish them.}
\end{enumerate}
\end{definition}}}
\chadded{Restriction~\ref{item:EmptyUnselected} is a simplifying assumption since the purpose of CI reporting procedures is to provide CIs for parameters in the selected set $S$. Hence, we just report the empty set for unselected parameters.
The above restriction (b) does not seem to be dispensable from the proof  of admissibility which we provide later (see Remark \ref{rem:morediscussion} in \Cref{sec:AdmissibilityProof}).
It is a reasonable assumption: all the CIs that are discarded or uninteresting should not affect how we report the selected CIs.
Both the e-BY and the weighted version of the e-BY procedure (which we introduce and describe in \Cref{sec:WeightedEBY}) satisfy this requirement.

In a statistical experiment, let $\Erm_i \coloneqq  \{E_i(\theta)\}_{\theta \in \Theta_i}$ denote the family of e-values used to construct e-CIs for $\theta_i^*$. $\mathbf{C}$ are realizations of the e-CIs corresponding to the
families of e-values $\mathbf{E} \coloneqq (\Erm_i)_{i \in [K]}$  
via
\begin{align}
\label{eq:associated-K}
C^{\mathbf E}_i(\alpha) \coloneqq C^{\Erm_i}(\alpha) = \left\{\theta \in \Theta_i: E_i(\theta)  < \frac{1}{\alpha}\right\} \text{ for each }\alpha \in [0,1],~i\in [K].
\end{align}}

Let $\mathbf{C}^{\mathbf{E}} \coloneqq (C_1^{\mathbf{E}}, \dots, C_K^{\mathbf{E}})$ denote the random vector of e-CIs that arise from $\mathbf{E}$.
Recall that a CI reporting procedure only has access to the realized selected set and the e-CIs, i.e., $S$ and $\mathbf C$, respectively but not the random e-CIs $\mathbf{C}^{\mathbf{E}}$ themselves.
The FCR of the CI reporting procedure $\Phi$ for the selection $\mathcal S$,  vector of e-value families $\mathbf E$, and true parameters $\theta_1^*\in\Theta_1,\dots,\theta_K^*\in\Theta_K$
is given by
\begin{align}
\mathrm{FCR}(\Phi)  & = \expect\left[\frac{\sum_{i\in \mathcal S} \ind{\theta_i^* \not \in \Phi(\mathcal S,\mathbf C^\mathbf E)_i }}{|\mathcal S| \vee 1}\right],
\label{eq:FCRdef}
\end{align}
where the expectation is taken under the true distribution $\pdist^*$.

A CI reporting procedure $\Phi$ has FCR  level $\delta\in [0,1]$ if $\mathrm{FCR}(\Phi)\le \delta$ for any selection $\mathcal S$, e-value families $\mathbf{E}$, and true values of parameters $\boldsymbol \theta^*$.

\begin{definition}[Dominance]
We say that the  CI reporting procedure $\Phi$ \emph{dominates} another one $\Phi'$
\chadded{ if $\Phi(S,\mathbf C)_i \subseteq \Phi'( S,\mathbf C)_i$ for each $i \in S$, i.e., $\Phi$ produces narrower CIs for all parameters in $S$ than the CIs produced by $\Phi'$, for all $S\subseteq [K]$ and $\mathbf C\in \mathrm{ECI}_K$,
and \emph{strictly dominates} if further there exists $i \in S$ s.t.\ $\Phi( S,\mathbf C)_i \subsetneq \Phi'( S,\mathbf C)_i$ for some}
 $S\subseteq [K]$ and   $\mathbf C\in \mathrm{ECI}^{*}_K$ where $\mathrm{ECI}^{*}_K$ is the  set of all $\mathbf C\in \mathrm{ECI}_K$ with strictly decreasing components.
 \end{definition}

Now, that we have a notion of dominance between e-CI reporting procedures, we can define admissibility.

\begin{definition}[Admissible CI reporting procedure]
A CI reporting procedure with FCR level~$\delta\in (0,1)$ is \emph{admissible} if it is not strictly dominated by any CI reporting procedure with FCR level $\delta$.
\end{definition}

\begin{theorem}\label{th:admissible}
The e-BY procedure at level $\delta\in (0,1)$ is an admissible CI reporting procedure with FCR level $\delta$.
\end{theorem}

We defer the proof of this theorem to \Cref{sec:AdmissibilityProof}. Further, using the same arguments as in the proof of \Cref{th:admissible}, we can also show that a weighted version of the e-BY procedure, defined in \Cref{sec:WeightedEBY}, is also admissible.

\section{Numerical simulations}
\label{sec:Simulations}
To understand the practical difference in the precision of the BY procedure and the e-BY procedure, we conducted simulations in two different data settings. The first is a nonparametric setting, where we only make the assumption that the data is bounded. In the second setting, we simulate the sharp FCR scenario for e-BY described in \Cref{sec:SharpFCR}, where the data are stopped Brownian motions.

\subsection{Bounded setting}
\label{sec:BoundedSimulations}

We wish to estimate \(K\) different means in this setting. For each \(i \in [K]\), let \(X_i \in [-1, 1]\) be a bounded random variable of interest, and let \(\mathbf{X}_i = (X_i^1, \dots, X_i^n)\) represent \(n\) i.i.d.\ draws from the distribution of \(X_i\). Thus, the \(i\)th parameter of interest is \(\theta_i^* = \expect[X_i]\). The distribution of \(X_i\) is standard normal distribution truncated to \([-\sigma, \sigma]\) that is normalized to be supported on \([-1, 1]\).
We set \(\sigma = 100\).

Let our desired level of FCR control be \(\delta = 0.1\). The selection rule selects all parameters with p-values less than \(\delta\). Let \(\widehat{\theta}_i\) denote the sample mean of \(X_i^1, \dots, X_i^n\) for each \(i \in [K]\). Our p-value of choice for the \(i\)th parameter is \chadded{\(P_i = (2 \exp(-\widehat{\theta}_i^2/(2n))) \wedge 1\)}. \chadded{Such $P_i$ is a two-sided p-value for bounded random variables derived from Hoeffding's inequality that tests the null hypothesis $H_i: \theta_i^* = 0$ for each $i \in [K]$.}

\begin{figure}[ht]
    \centering
    \includegraphics[width=\textwidth]{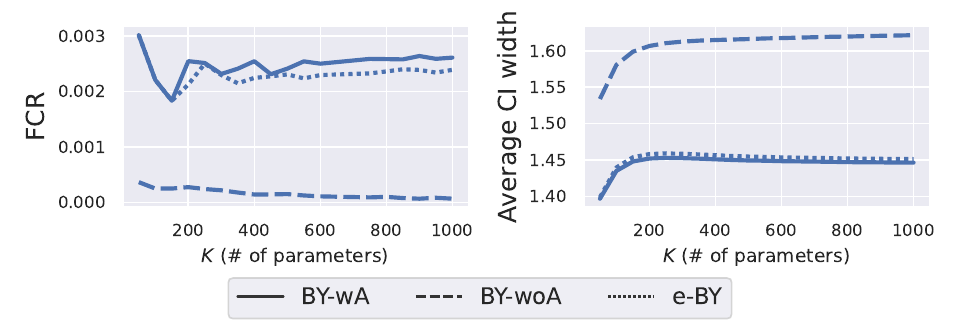}
    \includegraphics[trim={0 1cm 0 0},clip,width=0.5\textwidth]{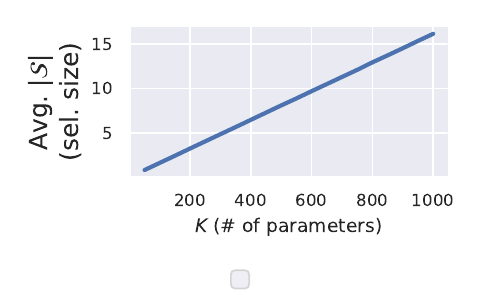}

    \caption{Comparison of the FCR and average CI width for the BY vs.\ the e-BY procedure in the bounded setting. \chadded{We also plot the average number of parameters selected (i.e., $|\selset|$) for each $K$.} ``BY-wA'' refers to the BY procedure under the assumptions that the CIs are independent (or PRDS) and  \(R_{i}^{\min} = |\selset|\), while ``BY-woA'' is with no assumptions. The average CI width of the e-BY procedure is nearly the same as the BY procedure with assumptions, and both are vastly tighter than BY with no assumptions.
    }
    \label{fig:Simulation}
\end{figure}

The CI for the BY procedure, and the e-CI for the e-BY procedure as defined as follows:
\begin{align}
C_i(\alpha) &\coloneqq \left(\widehat{\theta}_i \pm \sqrt{\frac{2\log(2 / \alpha)}{n}}\right),\\C_i(\alpha) &\coloneqq \left(\widehat{\theta}_i \pm \sqrt{\frac{2\log(2 / \alpha)}{n}} \cdot \frac{\log(2 / \alpha) + \log(4 / \delta)}{2\sqrt{\log(2 / \alpha)\log(4 / \delta)}}\right).
\end{align}
These two CIs are equivalent when \(\alpha = \delta / 2\) --- both are derived from Hoeffding's inequality. We compare three different post-selection inference methods: the e-BY procedure, the BY procedure with both an independence and \(R_i^{\min} = |\selset|\) assumption (which is satisfied in this setting), and the BY procedure without any assumptions. We refer to the BY procedure with the assumptions as ``BY-wA'', and the BY procedure without assumptions as ``BY-woA''.

\paragraph{Results} In Figure~\ref{fig:Simulation}, we see that e-BY and BY-wA have nearly the same expected width. On the other hand, BY-woA has a much larger expected width since it is uniformly dominated by BY-wA. This is a setting where one should use the BY-wA procedure for post-selection inference, since there is no dependence.
However, if the selection rule or dependence is unknown beforehand, one can safely use the e-BY procedure, and attain CIs that are nearly as tight as the ones obtained through BY-wA.

\subsection{Stopped Brownian motion setting}

We also simulate the sharp FCR setting from \Cref{sec:SharpFCR} for different true drift parameters, $\theta_i^* = \theta^*$ for all \(i \in [K]\), and number of parameters, $K$ to compare the e-BY and BY procedures under a setting where the dependence structure and selection rule is more complex. In this setting, the assumptions of the BY procedure are not satisfied --- the \(X_i\) are neither independent nor PRDS across \(i \in [K]\). Hence, BY-wA does not have guaranteed FCR control in this setting. We compare the procedures on choices of \(K \in \{10, 30, 100, 300, 1000\}\) and \(\theta^* \in \{10^{-1}, 10^{-2}, 10^{-3}\}\). We set \(\delta=0.05\) and \(\gamma=2\). The Brownian motion processes, \((W_i^t)\) for each \(i \in [K]\), are discretized into time steps of size \(10^{-5}\).
We use the e-CI formulated in \eqref{eqn:BrownianCI} for the e-BY procedure and its running intersection and for the BY procedure.
\begin{figure}[h]
    \centering
    \includegraphics[width=\textwidth]{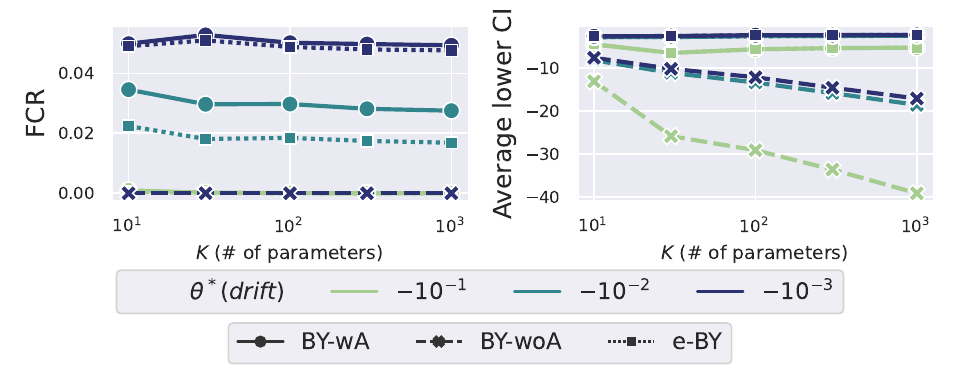}
    \includegraphics[trim={0 1cm 0 0},clip,width=0.5\textwidth]{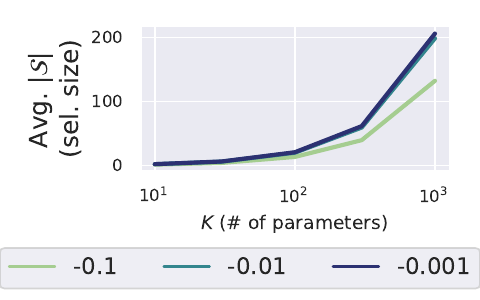}

    \caption{Comparison of the e-BY procedure vs.\ the BY procedure for the stopped Brownian motion setting. In this setting, all CIs produced are one-sided, and \(\theta^* < 0\), so we desire the average lower boundary to be as large as possible. The e-BY procedure and the BY-wA procedure have tight FCR as \(\theta^* \uparrow 0\) and vastly higher average lower boundaries than the the BY-woA procedure. BY-woA has extremely conservative (nearly 0) FCR control.}
    \label{fig:StoppedBrownianMotion}
\end{figure}

\paragraph{Results} In \Cref{fig:StoppedBrownianMotion}, we plot the empirical FCR and average CI lower bound (since the CIs are all one-sided) of each method for different values of the true parameter \(\theta^*\). Although it has no theoretical guarantees, BY-wA has an empirical FCR that is under \(\delta\) for all experiment parameters. The e-BY procedure, which does have an FCR guarantee, also has FCR control, and has an average lower CI endpoint that is nearly as large as BY-wA. On the other hand, BY-woA has an extremely conservative empirical FCR and produces much smaller average lower CI endpoint than the other two methods. This reflects the best-of-both worlds behavior of the e-BY procedure in practice --- e-BY has provable FCR control and CI widths comparable to the tightest possible widths produced by BY.

Thus, the e-BY procedure provides FCR control and CIs competitive with the BY procedure while requiring minimal assumptions. We will next show that this is also the case when using the e-BY procedure on real world data from A/B testing.

\section{Application: decision making in A/B testing}
\label{sec:ABTesting}

One natural application of post-selection inference is in A/B testing --- A/B testing is the use of randomized control trials to decide which features (among many) have a positive impact on users and should be shipped with the product. This method is typically used by information technology companies to determine whether a new version of a software product improves over the current version on certain user metrics.
Each A/B test is run in a sequential fashion, where users are continually being added to the experiment over time. The goal of an A/B test is to stop as quickly as possible and allow the data scientist to make a confident decision about which version of software to ship based on estimates of the population mean of the user metrics.
The e-BY procedure fits the A/B testing setup well for two key reasons: (1) the measurements of different metrics have a complex dependence structure since multiple measurements are made of a single user, and (2) the setup of the A/B testing is already sequential, so e-CIs, specifically stopped CSes, should already be the default for inference on each metric.

To simulate the behavior of data scientists choosing metrics to justify shipping decisions, we derive a selection method from the criteria for justifying shipping decisions of a single team within Twitter, a large information technology company. The shipping criteria provides guidelines for which combinations of metrics for which the difference between the treatment and control versions need to be statistically significant in a favorable direction, and which metrics that should not be significant in an unfavorable direction. We implemented a simplified version of this shipping criteria to be our selection method. Hence, our selection rule is the same across both procedures, and we are simply testing how the choice of post-selection inference procedure affects when the shipping criteria is considered satisfied.

\paragraph{A/B testing dataset from Twitter.} We compare the e-BY procedure vs.\ the BY procedure on a a dataset of actual A/B testing experiments from Twitter in a 1.5 year period and ran for at least two weeks. There are a total of 263 experiments in this dataset. Each of these experiments kept track of all 15 metrics ($K = 30$ since we treating the control and treatment versions of each metric as separate parameters) that were necessary for the shipping criteria at the daily level.

 Since the user data is collected sequentially in these experiments, we use the following CS from \citet{waudby-smith_time-uniform_central_2022} for the e-BY procedure:
\begin{align}
	C_t^i(\alpha) = \left(\widehat{\mu}_i^t\ \pm 1.7\sqrt{\widehat{\sigma}^2_{i, t} \cdot \frac{\log\log(2t) + 0.72\log(5.2 / \alpha)}{t}}\right),
\end{align} where \(\widehat{\mu}_i^t\) and \( \widehat{\sigma}^2_{i, t}\) are  the empirical mean and variance, respectively, of \(i\)th metric for the first \(t\) users.
This is an asymptotic CS , i.e.,\ the boundaries approach that of a true CS as the sample size increases under similar conditions as the central limit theorem for a traditional fixed-time asymptotic CI \citep{waudby-smith_time-uniform_central_2022}. Since the sample size in these experiments are on the order of $10^6$ users even on the first day, this asymptotic CS is very close to the true CS boundary.
We use \(\bar{C}_t^i = \cap_{i = 1}^t C_t^i\) , i.e.,\ the intersection of all intervals so far for the BY procedure. \begin{figure}[h]
    \centering
    \includegraphics[width=\textwidth]{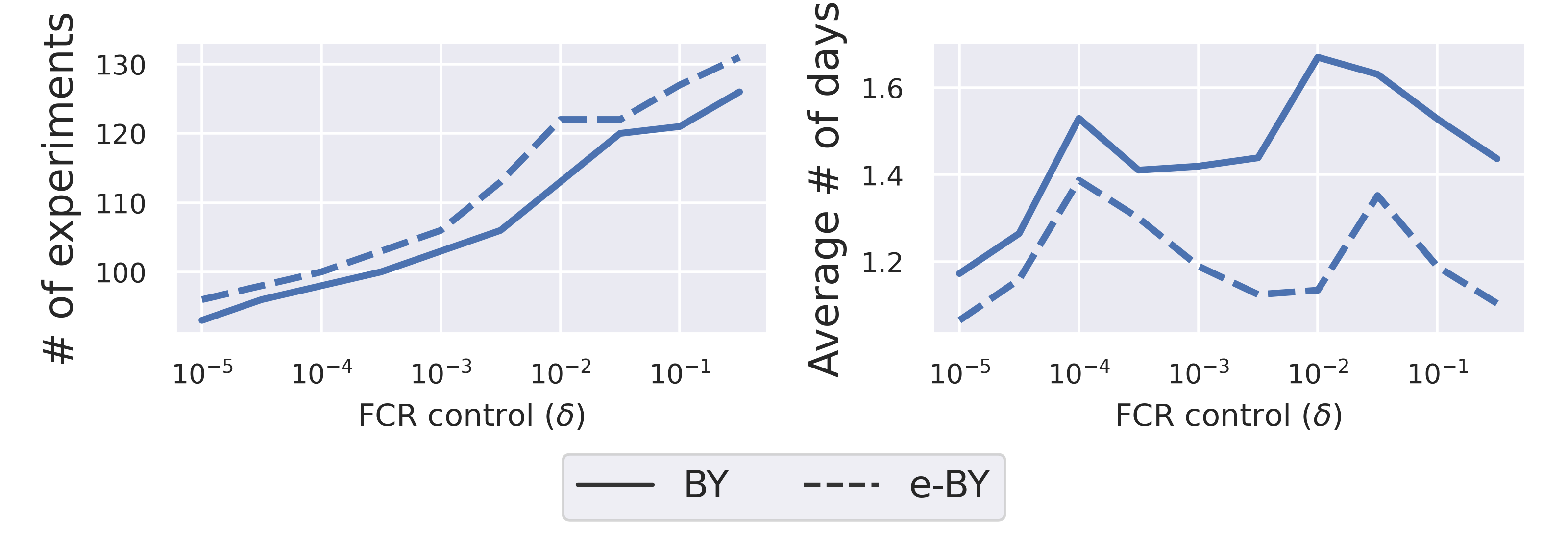}
    \caption{Results of the e-BY procedure vs.\ the BY procedure on the real Twitter A/B testing data across different levels of \(\delta \in [0, 1]\) for each procedure. The left figure depicts the total number of experiments where each procedure satisfied shipping criterion. The right figure depicts the average number of days an experiment ran before the procedure satisfied the shipping criteria. The set of experiments where the e-BY procedure satisfies shipping criteria is a strict superset of the set where the BY procedure satisfies shipping criteria (since e-BY dominates BY under dependence). Hence, the average number of days to satisfying the criteria is calculated over only the experiments for which the BY (and e-BY) procedure satisfied the shipping criteria.
    .}
    \label{fig:Results}
\end{figure}
\paragraph{Results.} The results of our analysis in \Cref{fig:Results} show that the e-BY procedure is better than the BY procedure in both number of experiments where a shipping decision can be justified, and the average time within each experiment to have sufficient evidence to satisfy the shipping criteria. For example, for a reasonable level of FCR control of \(\delta=0.1\),  the e-BY procedure justified shipping decisions for 127 experiments --- 5 more than 122 for the BY procedure. Note that the set of experiments where shipping was justified by the e-BY procedure is a superset of the experiments that the BY procedure justified shipping for. Consequently, we can also compare the average number of days before a shipping decision is justified for the e-BY and BY procedures, by taking the average over the experiments that the BY procedure justified shipping decisions on. We see that the the e-BY procedure took 1.2 days on average to satisfy the shipping criteria --- 0.3 days less than the BY procedure which took on average 1.5 days. Clearly, the tightness of the e-BY procedure makes a practical impact on the efficiency of A/B testing.

\section{Conclusion}

We have shown that the e-BY procedure is a versatile method for controlling FCR in post-selection inference and improves inference in a broad class of settings. While e-BY is restricted use only on e-CIs, universal inference and stopped confidence sequences~\citep{wasserman_universal_inference_2020,howard2020time,howard2021time} are already e-CIs. Further, the e-BY procedure maintains FCR control with no assumptions on the dependence structure or selection rule, and produces CIs with widths that are never larger (and usually smaller) than those of the BY procedure. In addition, we demonstrated that the e-BY procedure is optimal for e-CIs in some sense, as the FCR bound it guarantees is sharp, and that it is admissible in the domain of e-CI reporting procedures.

With respect to the utility of e-CIs, we observe  that most of the tightest CIs we can construct in nonparametric and sequential settings are already e-CIs. In the sequential setting, we showed that the e-BY procedure produces practically useful CIs for justifying shipping decisions on data from real A/B tests.
In addition, we discuss how to construct nontrivial e-CIs in any settings where CIs exist in \Cref{sec:Calibration}: we extend the p-to-e calibration methods introduced by \citet{vovk_logic_probability_1993b} to calibrating arbitrary CIs into e-CIs.

In this vein, an important direction for future study is to understand in what settings are the effectively tightest CIs also e-CIs. In many settings (e.g.\ Gaussian) where the exact (often asymptotic) distribution of a statistic under each parameter is known and identical across parameters, the tightest CIs (e.g.\ CI based on Gaussian CDF) are not e-CIs. On the other hand, e-CIs are often the only kind of CI we can construct when we are testing complicated composite nulls. Particularly in sequential settings where we desire an anytime-valid e-CI, \citet{ramdas_admissible_2020} have shown that e-CIs are nearly admissible. Hence, further study of e-CIs would improve the utility of the e-BY procedure.
 
\paragraph{Acknowledgements} We would like to thank Umashanthi Pavalanathan and Luke Sonnet for insightful discussions about the application of post-selection inference to A/B testing and their help in gathering and releasing the data from Twitter used in our experiments.
 \bibliography{eby}
\appendix
\section*{Appendix}
The appendix is organized as follows. We provide all the proofs we had omitted in the body of the main paper in \Cref{sec:Proofs}, in order of their mention in the paper. \Cref{sec:MultipleTesting} describes explicit connections between the post-selection inference procedures described in this paper for FCR control with FDR controlling procedures for multiple testing. We then present a weighted version of the e-BY procedure in \Cref{sec:WeightedEBY}. To complement our previous discussion of admissibility, we also describe an improvement we can make to the e-BY procedure in \Cref{sec:Improvement} if we discard some of the restrictions made on CI-reporting procedures. We highlight the difference between conventional Chernoff based CIs and Chernoff e-CIs in \Cref{sec:BatchChernoff} by introducing a notion of generalized e-CIs --- this notion was applied in our simulations in the bounded setting in \Cref{sec:BoundedSimulations}. Lastly, we discuss the advantages of FCR and how FCR relates to other error metrics that could be considered for the post-selection inference problem in \Cref{sec:CoverageTypes}.
\section{Proofs}
\label{sec:Proofs}

We produce the proofs that we had omitted in the main body of the paper in this section.

\subsection{Proof of \texorpdfstring{\Cref{thm:Calibration}}{calibration theorem}}
\label{sec:CalibrationProof}

We first note that \Cref{def:Calibrator} implies that \(f(P)\) is an e-value for any calibrator \(f\) and p-value \(P\) --- even if \(P\) is uniformly distributed.

To connect this idea with CIs, observe that the ``dual'' form of a CI is a p-value --- a CI represents a collection of hypothesis tests for a fixed type I error (i.e.,\ confidence level) across a set of hypotheses. In essence, it is the set of hypotheses that would not be rejected at some fixed level \(\alpha\) with the current realization of the sampled data. In contrast, a p-value represents a collection of hypothesis tests for a fixed hypothesis, but across different rejection levels \(\alpha\). Here, the p-value itself is the type I error of the test with the smallest rejection level that rejects the fixed null hypothesis.
\begin{definition}
    A \emph{p-value} \(P\) for a set of distributions \(\mathcal{Q}\) is a random variable supported on $[0, 1]$ that satisfies the following:
\begin{align}
    \sup_{\pdist \in \mathcal{Q}}\ \mathbb{P}_{\pdist}(P \leq \alpha) \leq \alpha \text{ for all } \alpha \in [0, 1].
\end{align}
\label{def:pvalue}
\end{definition}

Thus, we can calibrate the implicit p-value associated with every CI constructing procedure to an e-value, and then produce the e-CI associated with the e-value to calibrate from a CI to an e-CI.

Let \(C(\alpha)\) produce a \((1 - \alpha)\) CI for any \(\alpha \in [0, 1]\). Consequently, for a fixed parameter \(\theta \in \Theta\), the following is a p-value for the set \(\mathcal{Q}_\theta \coloneqq \{\pdist \in \Pcal: T(\pdist) = \theta\}\):
\begin{align}
    P^{\mathrm{dual}}(\theta) \coloneqq \inf \left\{\alpha \in [0, 1]: \theta \not \in C(\alpha)\right\}.
\end{align}

Consequently, \[E^{\mathrm{cal}}(\theta) \coloneqq f(P^{\mathrm{dual}}(\theta))\] forms a family of e-values. Hence, \[C^{\mathrm{cal}}(\alpha) \coloneqq \left\{\theta \in \Theta: E^{\mathrm{cal}}(\theta) < \frac{1}{\alpha}\right\}\] is a \((1 - \alpha)\)-e-CI. To show the equality between \(C^{\mathrm{cal}}\) and \(C\), we make the following observations.
\begin{align}
    C^{\mathrm{cal}}(\alpha) &= \left\{\theta \in \Theta: E^{\mathrm{cal}}(\theta) < \frac{1}{\alpha}\right\} = \left\{\theta \in \Theta: f(P^{\mathrm{dual}}(\theta)) < \frac{1}{\alpha}\right\} \\
    &\labelrel={rel:CalibUSC} \left\{ \theta: P^{\mathrm{dual}}(\theta) >  \max \left\{p: f(p) \geq \frac{1}{\alpha}\right\} \right\}\\
    &= \left\{ \theta: P^{\mathrm{dual}}(\theta) >  f^{-1}\left(\frac{1}{\alpha}\right) \right\} \\
    &\labelrel={rel:CIUSC}  C\left(f^{-1}\left(\frac{1}{\alpha}\right)\right).
\end{align} The equality at \eqref{rel:CalibUSC} is a result of \(f\) being nonincreasing and upper semicontinuous at \(1 / \alpha\). Hence the supremum is achieved and the equality holds. Similarly, the equality at \eqref{rel:CIUSC} is because \(C\) is nonincreasing and continuous from below at \(f^{-1}\left(1 / \alpha\right)\). If \(P^{\mathrm{dual}}(\theta) = f^{-1}\left(1 / \alpha\right)\), then \(\theta \not\in C_i(P^{\mathrm{dual}}(\theta))\) by \(C\) being continuous from below.

\subsection{Proof of \Cref{thm:Sharpness}}
\label{sec:SharpnessProof}

To prove the sharpness of the above situation, we require the following fact about Brownian motions.

\begin{fact}[{\citealt[p.223, equation 2.1.2]{borodin_handbook_brownian_1996}}]
Let \((W_t)_{t \in \reals^+_0}\) be a Brownian motion  with drift $\mu$ and initial value $x$. Define stopping times \(\tau^{a} \coloneqq \inf\{t \geq 0: W_t \geq x + a\}\), and \(\tau^{b} \coloneqq \inf\{t \geq 0: W_t \leq x - b\}\), where \(a, b > 0\). Then,
\begin{align}
\prob{\tau^a < \tau^b} =f(\mu,a,b) := 1-e^{-\mu b} \frac{\sinh (|\mu| a) }{\sinh(|\mu| (a+b))}.
\end{align}
\label{fact:StoppingProb}
\end{fact}
We note that $f(\mu,a,b)$ is increasing in $\mu$ since the larger the drift term, the more likely that the Brownian motion hits the upper boundary $x+a$.  Moreover,
\begin{equation}
\label{eq:FLimit}
 \lim_{\mu\uparrow 0}f(\mu,a,b)= 1-  \lim_{\mu\uparrow 0}  e^{-\mu b} \frac{e^{|\mu| a } - e^{-|\mu| a      }  }{e^{|\mu|(a+b)} - e^{-|\mu|(a+b)}} = 1- \frac{a}{a+b}=\frac {b}{a+b},
\end{equation}
which recovers the simple case in which $\mu=0$ (i.e., the case that the Brownian motion is a martingale). Only the above properties of $f$  (not its precise formula) will be used in the result below.

Using \Cref{fact:StoppingProb},
for $i\in [K]$,  the probability that $W_i $ hits $\gamma-1$ before hitting $-1$ is given  by
 $$
  \prob{W_i^{\tau_i} =   \gamma -1} =   f(\theta_i^*, \gamma-1,1).
 $$
We can compute the mean of $|\selset|$ as
\begin{align}
\expect[|\selset|] &= \expect\left [ \sum_{i\in [K]} \ind{W_i^{\tau_i}  =  \gamma-1  } \right]\\
&= \sum_{i\in [K]}  \pr( W_i^{\tau_i} =  \gamma -1) = \sum_{i\in [K]} f(\theta_i^*, \gamma-1,1).
\label{eq:SelectSize}
\end{align}
 For $i\in\selset$, using \Cref{fact:StoppingProb} again and noting that $W_i^{\tau_i} = \gamma-1$, we have
\begin{equation}\label{eq:beta} \prob{W_i^{\tau_i}= \beta-1 \mid \selset}
=  f(\theta_i^*, \beta-\gamma,\gamma).   \end{equation}

Let us suppose that $\theta_i^*<0$ for each $i\in [K]$ (but this is not known to the scientist); that is, $X^i(0)$ is   a strict supermartingale.
In this case,   $E_i(0)=\beta$ only if $E_i(\theta_i^*) >\beta$.

Recall that $C_i(\alpha)=   \{\theta \in \Theta : E_i(\theta) < 1 / \alpha\}$.
We can now compute FCR, using  \eqref{eq:beta}, as
 \begin{align}
    \mathrm{FCR}  & = \expect\left[\frac{\sum_{i\in \selset} \ind{\theta_i^* \not \in C_i( 1/ \beta) }}{|\selset| \vee 1}\right]\\&=\expect\left[   \expect\left[\frac{\sum_{i\in \selset} \ind{E_i(\theta_i^*)\ge \beta }}{|\selset|\vee 1}\mid \selset\right]  \right] \\&=\expect\left[ \sum_{i\in \selset}\frac{ \prob{E_i(\theta_i^*) \ge \beta \mid \selset}}{|\selset|\vee 1} \right]\\
    &\ge \expect\left[ \sum_{i\in \selset}\frac{ \prob{E_i(0) = \beta \mid \selset}}{|\selset|\vee 1} \right]
    =\expect\left[ \frac{ \sum_{i\in \selset} f(\theta_i^*, \beta-\gamma,\gamma)}{|\selset|\vee 1} \right] .
\end{align}
Obviously, FCR is a function of $\boldsymbol\theta^*:=(\theta_1,\dots,\theta_K)$.
We consider the situation where $\theta_i^*\uparrow 0$ (i.e.,\ \(\theta_i^*\) approaches 0) for each $i\in [K]$, denoted by  $\boldsymbol\theta^*\uparrow \mathbf 0$.
By applying \eqref{eq:FLimit}, we get
 $$ f(\theta_i^*, \beta-\gamma,\gamma) \uparrow \frac{\gamma }{\beta} = \frac{\delta |\selset| \gamma}{K},$$ as $\theta_i^*\uparrow 0$. Moreover, when $\boldsymbol\theta^*\uparrow \mathbf 0$,    \eqref{eq:FLimit} and \eqref{eq:SelectSize}  together yield
 $$
  \expect_{\boldsymbol\theta^*}[|\selset|] \uparrow \frac{ K}{\gamma},
 $$
where we emphasize that the expectation is taken with respect to $\boldsymbol\theta^* $ (which varies).
Using the above results and the monotone convergence theorem, we get
 \begin{align}
\lim_{\boldsymbol\theta^*\uparrow \mathbf 0} \expect_{\boldsymbol\theta^*}\left[ \frac{ \sum_{i\in\selset} f(\theta_i^*, \beta-\gamma,\gamma)}{|\selset|\vee 1} \right]
 &= \lim_{\boldsymbol\theta^*\uparrow \mathbf 0}  \expect_{\boldsymbol\theta^*}\left[ \frac{\delta |\selset| \gamma}{K} \ind{|\selset|>0} \right]\\
    &= \lim_{\boldsymbol\theta^*\uparrow \mathbf 0}  \expect_{\boldsymbol\theta^*}\left[\frac{\delta |\selset| \gamma }{K} \right] = \delta.
\end{align} Since the FCR is upper bounded by \(\delta\) as a result of using the e-BY procedure (\Cref{thm:eBY}), we have shown our desired statement.

\begin{remark}
From the analysis above, we can see that the value $\delta$ of FCR of e-BY  is sharp when $\boldsymbol\theta^*$ approaches $\mathbf 0$ from below.
 Since both convergences  in the above  computation are monotone in $\boldsymbol\theta^*$, we have, in general, that $\mathrm{FCR}\le \delta$ for $\boldsymbol\theta^* \le \mathbf 0$.
On the other hand, if some $\theta_i^*>0$, then $E_i(\theta_i^*) < E_i(0) \le \beta$, and hence $\prob{E_i(\theta_i^*)>\beta}=0$, leading also to a smaller FCR for the above procedure.
(If one chooses a sightly larger threshold $\beta>K / (\delta |\selset|)$, then the scenario $\boldsymbol\theta^*\approx \mathbf 0$ with $\boldsymbol\theta^*\ge 0$ also has an FCR close to $\delta$.)

The independence of the e-processes is only used in \eqref{eq:beta} in the second step. This condition can be weakened to the independence between the e-processes in step 1 and their increment processes in step 2. The independence assumption does not reduce the FCR, at least not in an obvious way.
\end{remark}
\subsection{Proof of \Cref{lemma:ECINonincreasing}}
\label{sec:ProofECINonincreasing}

We first show the ``only if" statement.
Suppose that $C\in \mathrm{ECI}(\Theta)$ with associated e-value $E$.
It is clear that $C$ is nonincreasing.
To show its continuity from above, note that $E(\theta) < 1 / \alpha$
if and only if $E(\theta) < 1 / \beta$ for some $\beta>\alpha$.
Therefore, we have $\bigcup_{\beta > \alpha} C(\beta) = C(\alpha)$, and thus $C$ is continuous from below.

We next show the ``if" statement.
Let the   function $t:\Theta\to [0,\infty]$ be given by
\begin{equation}\label{eq:deft}
t(\theta)= \sup\ \left\{\frac{1}{\alpha}: \alpha \in [0,1),~\theta \not \in C(\alpha)\right\},
\end{equation}
where the convention is $\sup\emptyset = 1$ (because we can always set $C(1)=\emptyset$ so that $\theta \not\in C(1)$ In the two extreme cases,  if $C(\alpha)$ is  $\Theta$ on $[0,1)$, then $t(\theta)=1$ for each $\theta$, and if $C(\alpha)$ is   $\emptyset$ on $[0,1)$, then $t(\theta)=\infty$ for each $\theta$.

Since $C$ is nonincreasing and continuous from below,   \eqref{eq:deft} implies \(\theta \not\in C(t(\theta))\). Thus, we get the following equivalence:
\begin{equation}\label{eq:equi-t}
t(\theta) <\frac{1}{\alpha} ~\Longleftrightarrow~ \theta \in C(\alpha).
\end{equation}
Let $U $ be a standard uniform random variable under each of the parameter $\theta\in\Theta$, and  define
\begin{equation}\label{eq:defe}
E(\theta)  = t(\theta)\ind{U < \frac{1}{t(\theta)}},~~\theta \in\Theta.
\end{equation}
It is clear that $E(\theta)\ge 0$ and $\expect_\theta[E(\theta)]= 1$, and \eqref{eq:associated} holds if $E(\theta)=t(\theta)$ for each $\theta\in \Theta$, which is a possible realization of $E$.

\subsection{Proof of \texorpdfstring{\Cref{th:admissible}}{admissibility theorem}}
\label{sec:AdmissibilityProof}

We fix a universal probability space \(\mathfrak{P}\) and parameter ranges $\Theta_1, \dots, \Theta_K$.

We will proceed to prove \Cref{th:admissible} by contradiction. Suppose that there exists a CI reporting procedure $\Phi$ which strictly dominates the e-BY procedure.
That means  for all $S \subseteq [K]$ and $\mathbf C\in \mathrm{ECI}_K$, $ \Phi(S,\mathbf C)_i \subseteq C_i(\delta|S|/K) $,
and
\begin{gather}
    \text{there exists some }S\subseteq [K], \mathbf C\in \mathrm{ECI}_K\text{ and }i^*\in S\\
    \text{ s.t.\ }\Phi(S,\mathbf C)_{i^*} \subsetneq C_{i^*}(\delta|S|/K).
\end{gather}
From now on, fix the above $(S,\mathbf C, i^*)$ --- these values will be key components for our construction, which will result in a contradiction.
Clearly, $S$ is not empty.

We assume that there exists random variables $B_{\selset}$ and $U$ in each $\pdist \in \Pcal$ (the existence is guaranteed by enlarging the probability space if needed) that are distributed as follows:
\begin{align}
    B_{\selset} \sim \text{Bern}(|S| / K)\text{ and }U \sim \text{Uniform}[0, 1],
\end{align} and we let $B_\selset$ and $U$ be independent of each other.

 We will show that
$ \mathrm{FCR}(\Phi) >\delta $ for some random selection $\selset$, vector of families of e-values $\mathbf E=(\{E_1(\theta)\}_{\theta \in \Theta_1},\dots,\{E_K(\theta)\}_{\theta \in \Theta_K})$ and true values of parameters $\vecb\theta^*=(\theta_1^*,\dots,\theta_K^*)$.
The main work is to construct such a setting.

First, we define the distribution of our selection rule $\selset$. Let $ T:=[K]\setminus S$ be the complement of $S$. We specify $\selset = S$ if $B_\selset = 1$ and $\selset = T$ otherwise. Hence,
under all vectors of possible parameter values in $\Theta_1\times\dots \times\Theta_K$, $\selset$ has the following distribution:
\begin{align}
    \p(\selset=S)=\p(B_\selset=1) = |S|/K, \text{ and }\p(\selset=T) = \p(B_\selset = 0) = |T|/K,
\end{align}
In the following, statements  for $i\in T$ can be ignored if $T$ is empty.

We briefly explain the main idea behind our construction. We first note that   any construction of $\mathbf E$ and $\vecb{\theta}^*$, if the $i$-th component of $\Phi$ is equal to that of e-BY for all $i\ne i^*$. Then,  for $i\in\selset\setminus \{i^*\}$,
\begin{align}
\p\left( \theta_i \not \in   \Phi(\selset,\mathbf C^\mathbf E)_i ,~ \selset =S    \right)  &=
\p\left( \theta_i \not \in   C_i^{\mathbf E} \left(\frac{\delta|\selset|}{K}\right),~ \selset =S    \right)\\
&\le \p\left( E_i(\theta_i) \ge \frac{K}{\delta|\selset|}\right)
\le \frac{\delta|\selset|}{K},
\end{align}
and  for $i\in T$,
\begin{align}
\p\left( \theta_i \not \in   \Phi(T,\mathbf C^\mathbf E)_i ,~ \selset =T      \right) &=
\p\left( \theta_i \not \in   C_i^{\mathbf E} \left(\frac{\delta|T|}{K}\right),~ \selset =T\right)\\
&\le \p\left( E_i(\theta_i)  \ge \frac{K}{\delta|T|}\right)
\le \frac{\delta|T|}{K}.
\end{align}
We will need to construct suitable $\mathbf E$ and $\vecb{\theta}^*$ such that all above   inequalities hold as  equalities approximately by taking a limit, and for the index $i^*$ we have
$$\p\left( \theta_{i^*} \not \in     \Phi(\selset,\mathbf C^\mathbf E)_{i^*} ,~ \selset =S  \right)   >  \frac{\delta|\selset|}{K}+c.
$$
where $c>0$ is a constant that does not shrink to $0$ when we take a limit.
If we are able to achieve the above, then the FCR of $\Phi$ will be close to $\delta +c/ |\selset|$.

Let us specify our choice of $\vecb{\theta}^*$. First, take a small number $\epsilon  \in (0,1)$ which will later shrink to $0$.
\begin{enumerate}[label=(\alph*)]
\item
We take $\theta^*_{i^*} \in  C_{i^*}\left(\delta|\selset|/K\right)  \setminus \Phi(\selset,\mathbf C)_{i^*} $
which is possible since $ \Phi(\selset ,\mathbf C )_{i^*} \subsetneq C_{i^*}\left(\delta|\selset|/K\right)$.
\item  For $i\in\selset\setminus \{i^*\}$, we
take $\theta^*_i \in C_{i}\left((1-\epsilon)\delta|\selset|/K\right) \setminus C_{i}\left(\delta|\selset|/K\right),$
which is possible since $\mathbf C$ is strictly decreasing.
\item  For $i\in T $, we
take $\theta^*_i \in C_{i}\left((1 - \epsilon)\delta|T|/K\right) \setminus C_{i}\left(\delta|T|/K\right).$
\end{enumerate}
Note that by this construction, we have  $\theta^*_i\not\in \Phi(\selset,\vecb C)_i  $ for each $i\in [K]$. Further, we can arbitrarily select the true values for each parameter, $\vecb{\theta}^*$, with the following construction of our universe of probabilities $\Pcal$. We can let our universe of distributions contain joint distributions over $((U, B_{\selset}, Y_1), \dots (U, B_{\selset}, Y_K))$ where $Y_i$ is some random variable whose distribution determines $\theta_i^*$ for each $i \in [K]$. Let $\pdist_{Y_i}$ denote the marginal distribution of $Y_i$ in $\pdist$. We define $\vartheta_i(\pdist) = \vartheta_{Y_i}(\pdist_{Y_i})$, where $\vartheta_{Y_i}: \Pcal_{Y_i} \rightarrow \Theta_i$ is a functional that maps from the universe marginal distributions of $Y_i$, $\Pcal_{Y_i}$, to the parameter value space $\Theta_i$.
Since the distribution of $B_{\selset}$ and $U$ are identical across any $\pdist \in \Pcal$, we can simply select some $\vecb\theta$ and let $\pdist^* = \pdist_{\vecb\theta}$, where $\pdist_{\vecb\theta}$ is some distribution in $\Pcal$ with true values $\vecb\theta$ for its parameters.

For $i\in [K]$, similarly to \eqref{eq:deft}-\eqref{eq:equi-t}, define
\begin{align}t_i(\theta)= \sup\left\{\frac{1}{\alpha}: \alpha \in [0,1),~\theta \not\in C_i(\alpha)\right\}; \\~\mbox{this gives }~
t_i(\theta) < \frac{1}{\alpha} ~\Longleftrightarrow~ \theta \in C_i(\alpha).
\label{eq:defti}
\end{align}
Next, we define some quantities concerning the true parameter \(\vecb{\theta}^* = (\theta^*_1, \dots, \theta^*_K)\). \begin{align}
r_S^* = \frac{K}{\delta |\selset|}, ~~t_S:=\max_{i\in\selset} t_i(\theta^*_i) \in \left [ r^*_S, \frac{ r^*_S   }{1 -\epsilon}\right) ,\\
t_S^*:=t_{i^*}(\theta^*_{i^*})<r^*_S, ~~\ell_S \coloneqq \frac{t_S - t_S^*}{r_S^*}.
\end{align}

The variables marked with an asterisk $(*)$ are not dependent on \(\epsilon\) --- only \(t_S\) and \(\ell_S\) depend on \(\epsilon\).

Let  $U$ be a random variable that is standard uniform  and independent of $\selset$ under all $\pdist_{\vecb{\theta}'}$, and the existence of such may be achieved by enlarging the probability space.
Define the random variables $U_S$ and $U_T$ indexed by \(\vecb{\theta}'\) as follows. \begin{align}
U_S(\vecb{\theta}') &\coloneqq \frac{|\selset|}{K}U\ind{\selset=S}\left(\ind{ U \geq \frac{K}{|\selset|t_S }} \vee \ind{\vecb{\theta'} \neq\vecb{\theta}^*} \right)\\
&\qquad + \left(\frac{|T|}{K}U+\frac{|\selset|}{K}\right)\ind{\selset=T} \\
U_T &\coloneqq \frac{|T|}{K}U\ind{\selset=T} + \left(\frac{|\selset|}{K}U+\frac{|T|}{K}\right)\ind{\selset=S}.
\end{align}

Note that \(U_S(\vecb{\theta}')\) is 0 with probability \(1/t_S\) and uniformly distributed between \(\left[1/t_S, 1\right]\) under \(\pdist_{\vecb{\theta}'}
\) otherwise if and only if \(\vecb{\theta}' = \vecb{\theta}^*\).
If
\(\vecb{\theta'} \neq \vecb{\theta}^*\), then under \(\pdist_{\vecb{\theta}'}
\), \(U_S(\vecb{\theta}')\) and \(U_T\) are both uniform random variables over \([0, 1]\).

We are ready to define our e-values.
We first define $E_{i^*}$, which is the most sophisticated object:
\begin{align}
E_{i^*}(\theta)  =~& t_{i^*}(\theta)\ind{U_S <  \frac{1}{t_{i^*}(\theta) \vee t_S}}\\
&+ r^*_S\ind{\frac{1}{t_S}\le U_S <  \frac{1 + \ell(\theta)}{t_S}} \ind{t_{i^*}(\theta)< t_S},
\end{align}
for all each $\theta \in\Theta_{i^*}$, where
$$\ell(\theta)\coloneqq\frac{t_S - t_{i^*}(\theta)}{r_S^*}.$$

Clearly \(E_{i^*} \geq 0\). It remains to show that, for each $\theta \in \Theta_{i^*}$, the expectation is bounded as \(\expect_\theta[E_{i^*}(\theta)] \leq 1\). We proceed casewise on \(t_{i^*}(\theta)\).
If $t_{i^*} (\theta) \ge t_S$, then $$ \expect_{\theta}  [ E_{i^*}(\theta)   ] \le \expect_{\theta} \left[ t_{i^*}(\theta)\ind{U_S <  \frac{1}{t_{i^*}(\theta)}}\right]=  1. $$
The inequality is simply by construction of \(U_S\) --- \(U_S\) is uniform if \(\theta\) is the true parameter and \(t_{i^*}(\theta) > t_S\). When \(t_{i^*}(\theta) = t_S\), the following is still true: \(\prob{U_S < 1/{t_S}} = 1/{t_S}\).  If $t_{i^*}(\theta) < t_S$, then
$$
 \expect_{\theta} \left[ E_{i^*}(\theta)  \right] = \frac{t_{i^*}(\theta)}{t_S}  +  \frac{r^*_S\ell(\theta)}{t_S}  =  \frac{t_{i^*}(\theta)}{t_S}  + \frac {t_S-t_{i^*}(\theta)} { t_S}
 =1.
 $$

 Note that the first equality is true again because \(U_S\) always satisfies \[\prob{U_S  < 1/{t_S}} = 1/{t_S}\] and is otherwise uniformly distributed among values greater than \(1/{t_S}\).
 Therefore, $E_{i^*}$ is an e-value for $\Theta_{i^*}$. Note that in our choice of true parameters \(\vecb{\theta}^*\),  we are always in the \(t_{i^*}(\theta^*_{i^*}) = t_S^* < t_S\) case.

The other e-values are defined as  $$
 E_i(\theta)  = t_i(\theta)\ind{U_S < \frac{1
 }{t_i(\theta)}},~~\theta \in\Theta_i, ~~~i\in\selset\setminus \{i^*\},
 $$
 and
 $$
 E_i(\theta)  = t_i(\theta)\ind{U_T< \frac{1}{t_i(\theta)}} ,~~\theta \in\Theta_i, ~~~i\in T.
 $$
It is straightforward to check that each $E_i(\theta) $ above is an e-value  for  $\theta \in \Theta_i$.

Now that we have defined our e-values, we will show there exists a distribution where the FCR under CI-reporting procedure \(\Phi\) will be strictly larger than \(\delta\).

Note that by construction of \(U_S\) and \(U_T\), we know the following:\begin{equation}\label{eq:usut}
\left\{U_S \le \frac{|\selset|}{K}\right\} = \{ \selset=S\}~\mbox{and}~\left\{U_T \le \frac{|T|}{K}\right\} = \{ \selset=T\}.
\end{equation}
By \eqref{eq:defti} and the construction of $(\theta^*_i)_{i\in [K]}$, we have  $t_{i^*}(\theta^*_{i^*}) <K/(\delta|\selset|)$ as well as
\begin{align}
&\frac{K}{\delta|\selset|}\le t_i(\theta^*_i) <\frac{K}{(1 -\epsilon)\delta|\selset|} &~\mbox{for $i\in\selset\setminus \{i^*\}$}\notag\\
\mbox{~~~and~~~}&\frac{K}{\delta|T|}\le t_i(\theta^*_i) <\frac{K}{(1 -\epsilon)\delta|T|} &~\mbox{for $i\in T$.}
\label{eq:impcond}
\end{align}
Putting the above ranges of $t_i(\theta^*_i)$ together,  using \eqref{eq:usut}, we get
\begin{align}
&\left\{U_S < \frac{1}{t_i(\theta^*_i)}\right\}\subseteq \{\selset=S\}&\mbox{for all $i\in\selset\setminus\{i^*\}$}\notag\\
\mbox{and}\ & \left\{U_T < \frac{1}{t_i(\theta^*_i)}\right\}\subseteq \{\selset=T\}
&\mbox{for all $i \in T$.}
\label{eq:impcond2}
\end{align}
Moreover,
As $\epsilon\downarrow 0$, we have $t_S \downarrow r^*_S$,
and
$$
\ell_S \to \ell_S^* \coloneqq  1 - \frac{t_S^*}{r_S^*} >0.
$$

The above construction of e-values leads to the important condition under true parameters \(\vecb{\theta}^*\), via \eqref{eq:defti},
\begin{equation}
\label{eq:alleffort}
U_S<\frac{1}{t_S} ~\Longrightarrow~ E_i(\theta^*_i) = t_i(\theta^*_i) \mbox{ for all $i\in\selset$} ~\Longrightarrow~
\mathbf C^{\mathbf E}_S  = \mathbf C_S.
\end{equation}

Recall that by construction, we have $\theta^*_i\not\in \Phi(\selset,\mathbf C)_i$ for each $i\in\selset$.
Therefore,   by using \eqref{eq:impcond2} and \eqref{eq:alleffort}, for $i \in S\setminus \{i^*\}$,
\begin{align}
\p\left( \theta^*_i \not \in \Phi(\selset,\mathbf C^\mathbf E)_i ,~ \selset =S \right)& \ge
\p\left(   \mathbf C^{\mathbf E} _S= \mathbf C_S,~ \selset =S \right)\notag\\
&\ge\p\left( U_S<\frac{1}{t_S} ,~ \selset =S \right)  = \frac{1}{t_S}.  \label{eq:iinS}
\end{align}
Further, by construction, we know that \(t_S^* < t_S\), so we have
\begin{align}
&\p\left( \theta^*_{i^*} \not \in \Phi(\selset,\mathbf C^\mathbf E)_{i^*} ,~ \selset =S \right)\\
& \ge
\p\left(   \mathbf C^{\mathbf E} _S= \mathbf C_S,~ \selset =S \right) +\p\left(  E_{i^*}(\theta) \ge r_S^*,~  \mathbf C^{\mathbf E} _S\ne \mathbf C_S,~ \selset =S \right)   \notag
\\&
 \ge  \frac{1}{t_S} +  \p\left( \frac{1}{t_S}\leq  U_S<\frac{1 + \ell_S}{t_S}\right)  = \frac{1 + \ell_S}{t_S}.
  \label{eq:istar}
\end{align}
Write
 $$r^*_T: = \frac{K}{|T| \delta } ,~\mbox{~and~}~t_T:=\max_{i\in T} t_i(\theta^*_i)  \in \left [ r^*_T, \frac{1}{1 -\epsilon} r^*_T \right).$$
 In case $|T|=0$ the above quantities are set to $\infty$.

For $i\in T$,
using $\Phi(\selset,\mathbf C^\mathbf E)_i\subseteq  C^\mathbf E_i (\frac{\delta|\selset|}{K})$ and \eqref{eq:impcond}-\eqref{eq:impcond2}, we have
\begin{align}
\p\left( \theta^*_i \not \in \Phi(T,\mathbf C^\mathbf E)_i ,~ \selset =T \right)& \ge
\p\left( \theta^*_i \not \in C^\mathbf E_i \left(\frac{\delta|T|}{K}\right),~ \selset =T \right) \notag
\\& =
\p\left( E_i(\theta^*_i) \ge   r^*_T,~ \selset =T \right) \notag
\\ & \ge
\p\left(E_i(\theta^*_i) \ge r^*_T,~ U_T< \frac{1}{t_T},~ \selset =T \right) \notag \\ & =
\p\left(  U_T< \frac{1}{t_T}  \right) =\frac{1}{t_T}. \label{eq:iinT}
\end{align}
Putting \eqref{eq:iinS}, \eqref{eq:istar} and  \eqref{eq:iinT}, together, we obtain
\begin{align}  &\expect\left[\frac{\sum_{i\in \selset} \ind{\theta^*_i \not \in \Phi(\selset,\mathbf C^\mathbf E)_i  }}{|\selset| \vee 1}\right]\\
&=
\expect\left[\frac{\sum_{i\in   S} \ind{\theta^*_i \not \in \Phi(\selset,\mathbf C^\mathbf E)_i  }}{| S| \vee 1}
\ind{\selset=S}     \right] \\
&\qquad + \expect\left[\frac{\sum_{i\in   T} \ind{\theta^*_i \not \in \Phi(T,\mathbf C^\mathbf E)_i  }}{|T| \vee 1}   \ind{\selset=T }    \right]
\\& \ge   \frac{1}{|\selset| }\left(\frac{1 + \ell_S}{t_S}\right) +   \frac{1}{|\selset| } \sum_{i\in   S\setminus\{i^*\}} \frac{1}{t_S} + \frac{1}{|T|  \vee 1} \sum_{i\in T} \frac{1}{t_T}
\\&= \frac{\ell_S}{|\selset|t_S}+ \frac{1}{t_S} + \frac{1 }{t_T}  .
\end{align}
 Note that as $\epsilon \downarrow 0$, $t_S\to r_S^*$,  $t_T\to r_T^*$, and $\ell_S\to \ell_S^* > 0$ by \(t_S^* < r_S^*\).
Therefore,
$$\lim_{\epsilon \downarrow 0}  \frac{\ell_S}{|\selset|t_S}+ \frac{1}{t_S} + \frac{ 1}{t_T} = \frac{\ell_S^*}{|\selset|r_S^*} + \frac{1}{r^*_S}+ \frac{ 1  }{r^*_T}
 > \frac{\delta|\selset| }{K } + \frac{|T|\delta}{K}=\delta.$$
Thus, the FCR of the procedure $\Phi$ for our constructed setting is large than $\delta$. This shows that the e-BY procedure is admissible.

\begin{remark}\label{rem:morediscussion}
From the proof of  Theorem \ref{th:admissible}, we can see that the restriction (b)
that  $\Phi(\selset,\mathbf C)   =  \Phi(\selset,\mathbf C')
 $
 if $\mathbf C_S=\mathbf C'_S$ is important, because on the event $\selset=S$,  we only require  $ \mathbf C^\mathbf E _S=\mathbf C_S$, but not $\mathbf C^\mathbf E=\mathbf C$.
In our construction, we actually have $E_i(\theta^*_i)=0$ for $i\in T$ on the event $\{\selset=S\}$. Recall that we need many inequalities in the FCR formula to be approximately equalities.
 The only way to produce precisely the event $\{\mathbf C^\mathbf E=\mathbf C,~\selset=S\}$ is through the event $\{E_i=t_i \mbox{ for all $i\in [K]$ and $\selset=S$}\}$, but for the approximation we need $\expect[E_i \ind{\selset=T}]$ for $t\in T$ to be arbitrarily close to $1$, so it does not seem to be possible if we ``waste'' $E_i$, $i\in T$ to take positive values on the event $\selset=S$.
\end{remark}

 \begin{remark}
The proof of Theorem \ref{th:admissible} also justifies that e-BY has sharp FCR.
To see this,  for $\Phi$ being e-BY and the setting constructed in the proof of Theorem \ref{th:admissible} (omitting the special  treatment for $i^*$), using \eqref{eq:iinS} and  \eqref{eq:iinT},  we get
\begin{align}
&\p\left(\theta^*_i \not \in \Phi(\selset,\mathbf C^\mathbf E)_i ,~ \selset =S \right) \ge 1/t_S\text{ for }i\in\selset,\\
&\p\left(\theta^*_j \not \in \Phi(\selset,\mathbf C^\mathbf E)_j ,~ \selset =T \right) \ge 1/t_T
 \text{ for }j\in T.
\end{align}
The rest is taking a limit, and we see that the e-BY procedure has an FCR arbitrarily close to $\delta$ for this setting.
   \end{remark}

\section{Multiple testing procedures based on e-BY}
\label{sec:MultipleTesting}

We can derive procedures for the multiple testing problem directly from post-selection inference procedures. We will discuss how that can be accomplished, and provide explicit examples of such derivations. In the multiple testing problem, we wish to identify as many parameters that lie outside of a fixed null hypothesis as possible. Formally, the goal is to determine whether \(\theta_i^* \in \Theta^0_i\) where \(\Theta^0_i \subseteq \Theta_i\) is the null hypothesis (as opposed to directly estimating \(\theta^*_i\)). Thus, for each \(i \in [K]\), we must output a decision of whether we reject the null hypothesis \(H^0_i: \theta^*_i \in \Theta^0_i\) or not. We denote set of indices where the null hypothesis is rejected as \(\rejset\). The null hypotheses in the rejection set (\(i \in \rejset\) and \(H^0_i\) is true) are called \textit{false discoveries}. Analogous to the FCR, we wish to control the \textit{false discovery rate} (FDR), which is the expectation of the \textit{false discovery proportion} (FDP) :
\begin{align}
    \FDP \coloneqq \frac{\sum\limits_{i \in \rejset} \ind{\theta_i^* \in \Theta^0_i}}{|\rejset|}, \qquad \FDR \coloneqq \expect[\FDP].
\end{align}

In \Cref{fig:MultipleTestingFlow}, we depict how the e-BY and BY procedures subsume results about FDR controlling multiple testing procedures as a special case. The multiple testing analog of the BY procedure is the Benjamini-Hochberg (BH) procedure \citep{benjamini_controlling_1995-2,benjamini_control_false_2001}, and BY showed that FCR control of the BY procedure implies FDR control of the BH procedure. The same implication between the e-BY procedure and the e-BH procedure \citep{wang_false_2020}.
Further, the e-BY procedure is not a naive application of the e-BH procedure to the post-selection inference problem. In fact, the e-BY procedure is a more powerful result than the e-BH procedure in the sense that \Cref{thm:eBY} implies FDR control for the e-BH procedure while the converse is not true. 

\begin{figure}[h]
    \centering
    \includegraphics[width=\textwidth]{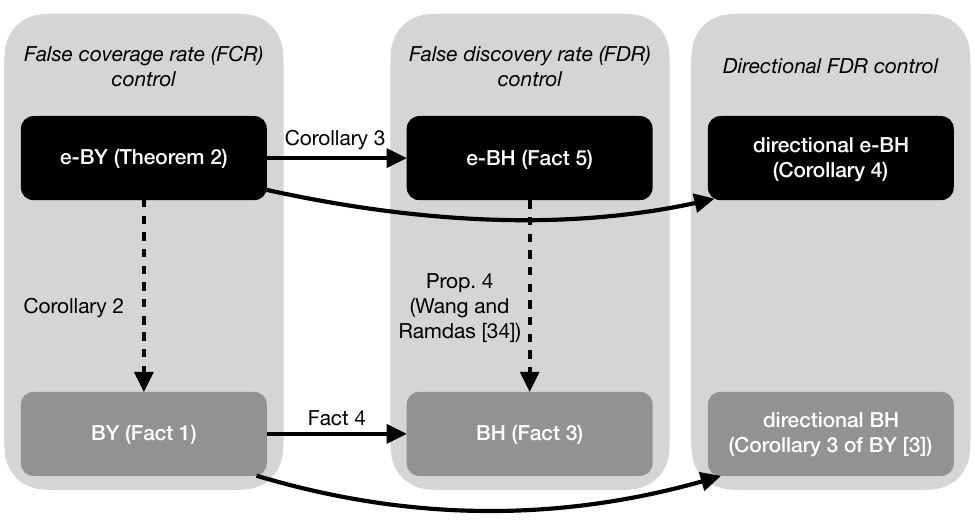}
    \caption{A diagram depicting the relationships between different procedures and their guarantees for FCR control in post-selection inference, and FDR control in multiple testing. The solid arrows signify that the target result is a special case of the source result. The dashed line signifies that the part of the target result that assumes arbitrary dependence between data is implied by the source result. This diagram illustrates how multiple testing procedures with FDR guarantees are downstream consequence of procedures for the post-selection inference problem with FCR guarantees.}
    \label{fig:MultipleTestingFlow}
\end{figure}
\subsection{Benjamini-Hochberg (BH) is a special case of BY}
\label{sec:BHtoBY}
BY showed that their procedure can be used to demonstrate FDR validity (i.e.\ provably controls FDR under some fixed level \(\delta\)) of the BH procedure. However, the BH procedure does not directly operate on CIs --- instead it takes as input \textit{p-values} i.e.\ \(P_i\) denotes the p-value for the \(i\)th null hypothesis for each \(i \in [K]\). \(P_i\) the superuniformity property under the null: \(\prob{P_i \leq \alpha} \leq \alpha\) for all \(\alpha \in (0, 1)\) if \(\theta^*_{i^*} \in \Theta_i^0\).

The BH algorithm proceeds as follows. Let \(P_1, \dots, P_K\) be the set of p-values computed from \(X_1, \dots, X_K\) for each hypothesis, and let \(P_{(1)}, P_{(2)}, \dots, P_{(K)}\) denote these p-values ordered from smallest to largest. Define a quantity \(D(K) \coloneqq K\) if the p-values corresponding to true null hypotheses are independent, or satisfy a positive dependence condition (see \citep{benjamini_control_false_2001}) and \(D(K) \coloneqq K\ell_K\) otherwise.
\begin{align}
k_{\BH} &\coloneqq
\max \left\{k \in [K]: P_{(k)} \leq \delta k / D(K)\right\} \cup \{0\},\\
\rejset_{\BH}&\coloneqq \left\{i \in [K]: P_i \leq  \delta k_{\BH} / D(K)\right\}.
\end{align}
\begin{fact}[BH procedure controls FDR \citep{benjamini_controlling_1995-2,benjamini_control_false_2001}]
The BH procedure, which rejects the hypotheses indexed by \(\rejset_{\BH}\), ensures \(\FDR \leq \delta\).
  \label{fact:BH}
\end{fact}
BY observed that the BH procedure can actually be formulated completely in terms of CIs and the BY procedure in the following fashion.
\begin{fact}[Corollary 2 of BY]
Let \(\selset = \rejset_{\BH}\), and define the CIs as follows:
\begin{align}
    C_i(\alpha) = \begin{cases}
\Theta_i \setminus \Theta^0_i & \text{if }P_i \leq \alpha\\
\Theta_i & \text{if }P_i > \alpha
\end{cases}.
\end{align}

Recall that, by definition of \(\rejset_{\BH}\), \(P_i \leq \delta k_{\BH}/D(K)\) for all \(i \in \rejset_{\BH}\). When we apply the BY procedure, we construct exactly \(C_i(\delta k_\BH/D(K)) = \Theta_i \setminus \Theta^0_i\) for each \(i \in \rejset_{\BH}\). As a result, a false discovery is made for the \(i\)th null hypothesis \((\theta^*_{i^*} \in \Theta_0)\) if and only if \(\theta^*_{i^*} \not\in C_i(\delta k_\BH/K)\). Hence \(\FDP = \FCP\) and \(\FDR = \FCR\). Consequently, the FCR guarantee of the BY procedure in \Cref{fact:BY} implies the BH guarantee of \(\FDR \leq \delta\) in \Cref{fact:BH}.
\label{fact:BHtoBY}
\end{fact}

Clearly, the BY procedure is a more powerful and general technique than the BH procedure, which is essentially the BY procedure for a specific choice of CI and selection rule. Hence, a guarantee on FDR by the BH procedure does not imply any guarantee on the FCR the BY procedure, while the reverse implication \emph{does} hold. We will see this paradigm also play out between the e-BY procedure and its multiple testing analog, the e-BH procedure.

\subsection{e-BH is a special case of e-BY}\label{sec:EBHtoEBY}
We will show this implication by reducing the multiple testing problem to an instance of the post-selection inference problem. The e-BH procedure takes e-values \(\mathbf{E} \coloneqq (E_1, \dots, E_K)\) as input, where \(E_i\) is an e-value w.r.t.\ \(\Theta^0_i\) for each \(i \in [K]\). Note that both e-BH and e-BY impose no requirements about the dependence between the e-values and e-CIs, respectively. Hence, \(E_1, \dots, E_K\) may be arbitrarily dependent.

\citet{wang_false_2020} show that FDR is controlled not only for e-BH, but for an entire class of ``self-consistent'' procedures. For a multiple testing procedure \(\mtproc: [0, \infty)^K \mapsto 2^{[K]}\) that takes e-values and outputs a rejection set \(\rejset\), \(\mtproc\) is \textit{self-consistent} \citep{blanchard_two_simple_2008} if and only if it satisfies the following property:
\begin{gather}
\begin{aligned}
    \text{For any realization of e-values } \mathbf{E} \in [0, \infty)^K, E_i \geq \frac{K}{\delta |\mtproc(\mathbf{E})|}\text{ for every }i \in \mtproc(\mathbf{E}).
\end{aligned}
    \label{eq:SelfConsistent}
\end{gather}
The e-BH procedure outputs the largest self-consistent set --- consequently, any other self-consistent set will be a subset of the rejection set output by the e-BH procedure.

\begin{fact}[FDR control of self-consistent procedures \citep{wang_false_2020}]
Any self-consistent procedure \(\mtproc\) that operates on e-values, i.e., that satisfies \eqref{eq:SelfConsistent}, ensures \(\FDR\leq\delta\).
\label{fact:EBH}
\end{fact}

\begin{corollary}[FCR control for e-BY implies FDR control for any self-consistent procedure]

Let \(\mtproc\) be a self-consistent multiple testing procedure, and define the selected parameters to correspond to the rejection set: \(\selset = \rejset = \mtproc(\mathbf{E})\). Define a family of e-values \(E_i^{\mathrm{ind}}(\theta) \coloneqq E_i \cdot \ind{\theta \in \Theta_i^0}\) and let the e-CIs corresponding to this family be:
\begin{align}
C_i^{\mathrm{ind}}(\alpha) \coloneqq \left\{\theta \in \Theta_i: E_i^{\mathrm{ind}}(\theta) \leq \frac{1}{\alpha}\right\} = \begin{cases}
\Theta_i \setminus \Theta_i^0 & \text{if }E_i > 1/\alpha\\
\Theta_i & \text{if }E_i \leq 1/\alpha
\end{cases}.
\end{align} The e-BY procedure outputs \(C_i^{\mathrm{ind}}(\delta |\mathcal{D}(\mathbf{E})|/K) = \Theta_i \setminus \Theta_i^0\) for each \(i \in \mathcal{D}(\mathbf{E})\). Thus, the e-BH procedure makes a false discovery if and only if the e-BY procedure does not cover a CI in \(\selset\). Hence, we have FDR control of the e-BH procedure through \Cref{thm:eBY}.

\label{corollary:eBHtoeBY}
\end{corollary}

Evidently, the e-BY procedure is more general than the e-BH procedure. We show in \Cref{sec:FSR} that we can use the e-BY procedure to also get control on a directional form of FDR that requires the procedure to output a sign or direction along with each rejection.

\subsection{Controlling the directional false discovery rate dFDR}
\label{sec:FSR}

We can also use the the e-BY procedure to provide guarantees beyond FDR for e-BH. Consider a specialized case of the multiple testing problem, where a scientist is performing only two-sided hypothesis tests e.g.\ discovering how the presence of certain genotypes have a significant effect on the baseline level of a certain hormone. The scientist may not have an a priori idea of which direction (positive or negative) the effect on a hormone a genotype could have. Thus, when rejecting the null hypothesis and making a discovery for some genotype-hormone interaction, it would also be invaluable to produce an estimate of the \textit{direction} of the effect. This problem of determining the direction along with significance was initiated by \citet{bohrer_multiple_threedecision_1979a,bohrer_optimal_multiple_1980}, and \citet{hochberg_multiple_classification_1986}. Building on this, BY showed that the BH procedure could ensure control of a directional version of the FDR when the operating on p-values for two-sided hypothesis tests through their FCR result for the BY procedure.

In this section, we describe how we can use e-BY to develop a direction determining version of e-BH, and also ensure control of a directional version of the FDR. For a notion of direction to be well-defined, we consider only parameter spaces where \(\Theta_i \subseteq \reals\), and for each parameter \(\theta_i^*\), we are testing the null hypothesis \(H_i^0: \theta_i^* = 0\). In addition to outputting a rejection set, we also require the multiple testing procedure to also assign a direction \(D_i \in \{\pm 1\}\) for each \(i \in \rejset\). Thus, we can define the directional FDR (\(\dFDR\)):
\begin{align}
    \dFDR \coloneqq \expect\left[\frac{\sum\limits_{i \in \rejset} \ind{D_i = 1, \theta_i^* \leq 0} + \ind{D_i = -1, \theta_i^* \geq 0}}{|\rejset|}\right].
\end{align} This definition of \(\dFDR\) coincides with the definition of mixed directional FDR defined in BY. The \(\dFDR\) is also nearly the same as the \textit{false sign rate} (\(\FSR\)) that is analyzed in \citet{stephens_false_discovery_2017,weinstein_online_control_2020} and the ``pure directional FDR'' discussed in BY. However, we will only consider the \(\dFDR\) since it is always larger than the other two directional error metrics. Consequently, these all of these metrics are equivalent when no \(\theta_i^*\) is exactly 0 --- \citet{tukey_philosophy_multiple_1991} argues happens this occurs in most realistic scenarios as many parameters can be off from 0 by an imperceptible amount so the distinction between these metrics may not be practically critical. We will also require that the input e-values are ``two-sided'' e-values. Specifically, let \(E_i \coloneqq (E_i^+ \vee E_i^-) / 2\), where \(E_i^+\) is an e-value w.r.t.\  \([0, \infty)\) and \(E_i^-\) is an e-value w.r.t.\  \((-\infty, 0]\), and two constituent e-values are inverses of each other \(E_i^+ = (E_i^-)^{-1}\). 

\begin{corollary}[e-BH with two-sided e-values controls \(\dFDR\)]
Let the selection set still be the output of the e-BH procedure , i.e.,\ \(\selset = \rejset_{\EBH}\). In addition, define the following direction decisions for each \(i \in \selset\):
\begin{align}
    D_i \coloneqq \begin{cases}
    -1 & \text{if }E_i^+ \geq 2K / (\delta k_\EBH)\\
    1 & \text{else}.
    \end{cases}
\end{align}

Define a family of e-values as so: \(E^{\mathrm{dir}}(\theta) \coloneqq E_i^+ \ind{\theta \geq 0} + E_i^- \ind{\theta \leq 0}\). We derive the following e-CI procedure from this family:
\begin{align}
    C_i^{\mathrm{dir}}(\alpha) \coloneqq \begin{cases}
        (-\infty, 0) & \text{if } E_i^+ \geq 2 / \alpha\\
        (0, \infty) & \text{if } E_i^- \geq 2 / \alpha\\
        \reals & \text{else}.
    \end{cases}
\end{align} Consequently, miscoverage occurs if and only if the direction of \(\theta_i^*\) for each \(i \in \selset\). Thus, the FCR guarantee from \Cref{thm:eBY} ensures that \(\dFDR \leq \delta\).
\end{corollary}

As a result, we have shown that e-BY procedure can be used to provide results for more general problems than just e-BH in the multiple testing scenario --- e-BY also provides \(\dFDR\) control in this directional variant of the multiple testing problem.
Thus, the FCR control provided by e-BY allows it to be useful, general tool for providing error control guarantees in a variety of problems.

\section{Weighted e-BY: weights for individual control of the size of each e-CI}
\label{sec:WeightedEBY}

The e-BY procedure in \Cref{def:eBY} assigns a equal confidence level to the CI constructed for each \(\theta_i^*\), but there may be cases where the user would desire tighter CIs for some parameters, and looser CIs for others. For example, the parameters we are estimating may be the severity of different side effects caused by a drug candidate. In that situation, we want to estimate more precisely the side effects that are life threatening or severely injurious to the recipient of the drug, and be willing to have larger CIs for estimating side effects that only cause minor ailments. We present a simple version of the weighted e-BY procedure where the weights are fixed beforehand. However, \citet{ignatiadis_e-values_unnormalized_2022} study the weighted e-BY procedure much more deeply and show that, surprisingly, the weights do not need to be normalized to sum to $K$ --- one may require the weights to be e-values themselves and still achieve valid FCR control. Further, they show that this approach of using e-values as weights is applicable to a weighted form of the BY procedure as well.

Let \(w_1, \dots, w_K\) be nonnnegative weights assigned to each of the \(K\) parameters and let their sum be bounded by \(K\). We can now prove a FCR guarantee about a weighted version of e-BY where the confidence level corresponding to the CI constructed for each \(\theta_i^*\) is \(1 - w_i \delta |\selset|/K\).

\begin{algorithm2e}[t]
\caption{The weighted e-BY procedure for constructing post-selection CIs with FCR control. Let \(w_1, \dots, w_K\) be nonnegative weights that sum to \(K\). The resulting constructed CIs have \(\FCR \leq \delta\), where \(\delta \in (0, 1)\) is a predetermined  level of error.\label{alg:WeightedEBY}}
\SetAlgoLined
{\nl Sample data \(\mathbf{X}\coloneqq (X_1, \dots, X_K)\) for each parameter being estimated --- these may be dependent.\\}
\nl
Let \(C_i\) produce \textit{marginal e-CIs} at confidence level \((1 - \alpha)\) for any \(\alpha \in (0, 1)\) for each \(i \in [K]\).

\nl Select a subset of parameters based on the data for CI construction: \(\selset \coloneqq
\selalg(\mathbf{X})\).\hfill \\
\nl
Set \(\alpha_i \coloneqq w_i \delta |\selset|/K\) for each \(i \in \selset\).\\
\nl Construct the CI \(C_i(\alpha_i)\) for \(\theta_i^*\) for each \(i \in \selset\).
\end{algorithm2e}

\begin{theorem}[Weighted e-BY controls FCR]
Let \(w_1, \dots, w_K\) be a set of nonnegative weights that sum to at most \(K\). The weighted e-BY procedure formulated in \Cref{alg:WeightedEBY} ensures \(\FCR \leq \delta\).
\end{theorem}
\begin{proof}
The proof follows similarly to that of \Cref{thm:eBY}.
\begin{align}
    \text{FCR} &= \mathbb{E}\left[ \frac{\sum_{i \in \selset} \ind{\theta_i^* \notin C_i\left(w_i\delta|\selset|/K\right)} }{|\selset| \vee 1} \right]\\
&= \mathbb{E}\left[ \frac{\sum_{i \in [K]} \ind{E_i(\theta_i^*)w_i  \delta |\selset|/K > 1} \cdot \ind{i \in S} }{|\selset| \vee 1} \right]\\
    &\leq \sum_{i \in [K]} \mathbb{E}\left[ \frac{E_i(\theta_i^*) w_i \delta |\selset|} {K(|\selset| \vee 1)} \right] =  \sum_{i \in [K]} \frac{w_i\delta}{K} \mathbb{E}\bigg[E_i(\theta_i^*)  \cdot \frac{ |\selset|}{|\selset| \vee 1} \bigg] \leq \delta.
\end{align} Here, the last inequality is a consequence of \(\sum\limits_{i = 1}^K w_i \leq K\).
\end{proof}

\section{Improving e-BY when a CI-reporting condition is relaxed}
\label{sec:Improvement}

We briefly explain the reason requiring strictly decreasing components of $\mathbf C\in \mathrm{ECI}^{*}_K$  for strict domination in \Cref{sec:CIreport}.
Suppose that $K=2$, $S=\{1,2\}$ and
$C_1(0)=C_1(1)$ (i.e., we are certain about $\theta_1$ in a subset of $\Theta_1$, but no further information), and $C_2 $ is strictly decreasing.
The e-BY procedure   reports $(C_1(\delta),C_2(\delta))$.
Consider another procedure $\Phi$ which   reports  $(C_1(0),C_2( \delta+\epsilon)) \subsetneq(C_1(\delta),C_2(\delta))$ for this specific $(C_1,C_2)$ and
behaves as e-BY for all other CIs. Clearly, this procedure strictly dominates e-BY.
Intuitively, it should be possible to choose $\epsilon>0$ small enough so that the FCR of $\Phi$ is less than or equal to $\delta$ (but we did not find a proof).  Generally, having a constant $C_i(\alpha)$ for $\alpha$ in an interval  leads to the above situation.
This situation is not interesting for practice as often an e-CI is strictly decreasing. but it becomes relevant in case the parameter $\theta_i^*$ takes finitely many values. We discuss how we could improve on the e-BY procedure if the CIs are not strictly decreasing (and restriction \ref{item:SelectionOnly} is lifted) in \Cref{sec:Improvement}.

In this section we explore a minor improvement of e-BY by relaxing the restriction \ref{item:SelectionOnly} in \Cref{sec:CIreport} in the formulation of a CI reporting procedure, as well as the strict decreasing condition of the e-CIs.

For $(\selset,\mathbf C)\in 2^{[K]} \times \mathrm{ECI}_K$  where $S\ne \emptyset, [K]$,
we define
\begin{align}
\gamma(\selset,\mathbf C)&:=\sum_{i\in\selset}\sup_{\theta\in \Theta_i} t_i(\theta ),\\
\lambda(\selset,\mathbf C)&:=\sum_{i\in [K]\setminus S}\inf_{\theta\in \Theta_i} \left(t_i(\theta )\ind{t_i(\theta)>1}\right),\\
v (\selset,\mathbf C)&:= 1+\frac{ \lambda (\selset,\mathbf C)}{\gamma(\selset,\mathbf C)} ,
\end{align}
where   $t_i$  is given in \eqref{eq:defti}, and we set $\infty/\infty=0$.
Note that $\gamma(\selset,\mathbf C)<\infty $ if and only if for each $i \in S$ there exists $\alpha_i>0$  such that $C(\alpha_i)=\Theta_i$; similarly,
$\lambda(\selset,\mathbf C)>0$ if and only if there exists $i\in [K]\setminus S$ such that
$
C_i (\alpha_i) = \emptyset
$ for some  $\alpha_i\in (0,1)$.
The condition $v (\selset,\mathbf C)>1$ is equivalent to $\gamma(\selset,\mathbf C)<\infty $ and $\lambda(\selset,\mathbf C)>0$. (Admittedly, this is an uncommon situation.)

For a level $w \in (1,\infty)$,    we denote by
$$\Sigma(w): =\{ (\selset,\mathbf C)\in 2^{[K]} \times \mathrm{ECI}_K: v(\selset,\mathbf C) \ge  w\}.$$
One can check that the set $\Sigma(w) $ is not empty.
In particular,
if $S=\{1\}$ and $\mathbf C$ is the constant CI of $\Theta_1\times\emptyset ^{K-1}$ (i.e., no information for the first one, and sure false coverage for all others), then $t_1(\theta)=1$ for all $\theta \in \Theta$, and $t_i(\theta)=\infty$ for all $\theta \in \Theta_i$ and $i\ne 1$. In this case, $\gamma(\selset,\mathbf C)=1$, $\lambda (\selset,\mathbf C)=\infty$, and $v(\selset,\mathbf C)=\infty$.

For a given level $w\in (1,\infty)$, define a CI reporting procedure $\Phi$ by
\begin{enumerate}
\item if the input $(\selset,\mathbf C)$ is in $\Sigma(w)$, then report e-BY at level $w \delta$;
\item otherwise, report e-BY at level $  \delta$.
\end{enumerate}
Clearly, the above procedure   dominates e-BY at level $\delta$, and the domination is strict noting that $\Sigma(w)$ is not empty.
If we allow $w=1$, then it is precisely the e-BY procedure.
\begin{proposition}
The above procedure $\Phi$ has FCR at most $\delta$.
\end{proposition}
\begin{proof}
Let $\selset$ be a random selection, $\mathbf E$ be the vector of e-values, and $\vecb{\theta}^*$ be the true parameter.
Denote by $A=\{(\selset,\mathbf C^\mathbf E)\in \Sigma(w)  \}$, that is, the event that     $v(\selset,\mathbf C^\mathbf E) \ge w$.
Note that by Lemma \ref{lemma:ECIEvalue}, if $\mathbf C^\mathbf E=\mathbf C$ occurs, then $E_i(\theta)=t_i(\theta)$ for all $\theta \in \Theta_i$ and $i\in [K]$ whenever $E_i(\theta)>1$,
where $t_i$ is given by \eqref{eq:defti}, that is,
$$ t_i(\theta)= \sup\left\{\frac{1}{\alpha}: \alpha \in [0,1),~\theta \not\in C_i(\alpha)\right\}.$$
In case $(\selset,\mathbf C^\mathbf E)=(\selset,\mathbf C)\in \Sigma(w)$, we have
\begin{align}
&\sum_{i\in    S}  (w-1)     E_i(\theta_i)  -
\sum_{i\in [K]\setminus  S}    E_i(\theta_i) \\
\le\ &  \sum_{i\in     S}  (w-1)  \sup_{\theta_i\in\Theta_i} t_i(\theta_i)
-   \sum_{i\in [K]\setminus  S}  \inf_{\theta_i\in\Theta_i}  t_i(\theta_i) \ind{t_i(\theta_i)>1}
\le 0.
\end{align}
Therefore, we have
\begin{equation}
\left(  \sum_{i\in     \mathcal  S}  (w-1)     E_i(\theta_i)  -
\sum_{i\in [K]\setminus \selset}    E_i(\theta_i) \right) \ind{A} \le 0.
\label{eq:leqzero}
\end{equation}
Using the above facts, we can compute the following upper bound on the FCR:
\begin{align}
&\mathrm{FCR}(\Phi)  \\
& = \expect\left[\frac{ \sum_{i\in \selset} \ind{\theta_i \not \in \Phi(\selset,\mathbf C^\mathbf E)}}{|\selset|\vee 1}\right]\\
&= \expect\left[\frac{ \sum_{i\in\mathcal   S} \ind{E_i(\theta_i) \ge K/(w \delta |\selset|)}}{| \selset|\vee 1}\ind{A} \right]\\
&\qquad+   \expect\left[\frac{ \sum_{i\in \selset} \ind{E_i(\theta_i) \ge K/(  \delta |\selset|)}}{|\selset|\vee 1}(1-\ind{A})\right]
\\& \le \expect\left[\frac{ 1}{| \selset|\vee 1}\sum_{i\in  \selset} \frac{w \delta |\selset| E_i(\theta_i)}{K} \ind{A} \right]  +   \expect\left[\frac{1}{|\selset|\vee 1} \sum_{i\in \selset}  \frac{ \delta |\selset| E_i(\theta_i) }{K}(1-\ind{A})\right]
\\& \le  \expect\left[   \frac{1}{K }   \sum_{i\in     \mathcal  S}  w \delta   E_i(\theta_i) \ind{A}\right]+
\frac{\delta}{K }
\expect\left[ \sum_{i\in [K]}       E_i(\theta_i)  (1-\ind{A})\right] .
\end{align}

Now, we can group the terms that involve \(\ind{A}\) together to simplify our bound on the FCR, and piece together the term we see in \eqref{eq:leqzero}.
\begin{align}
\FCR(\Phi)& \le
\frac{\delta }{K }  \sum_{i\in [K]}    \expect\left[      E_i(\theta_i)  \right]
+  \frac{\delta }{K } \expect \left [  \left(\sum_{i\in     \mathcal  S}  w     E_i(\theta_i)  -
\sum_{i\in [K]}    E_i(\theta_i) \right) \ind{A}
\right]
\\& \le
\delta
+  \frac{\delta }{K } \expect \left [  \left(  \sum_{i\in     \mathcal  S}  (w-1)     E_i(\theta_i)  -
\sum_{i\in [K]\setminus \selset}    E_i(\theta_i) \right) \ind{A}
\right]\le \delta,
\end{align}
where in the last inequality we used \eqref{eq:leqzero}.
Therefore, $\mathrm{FCR}(\Phi)  \le \delta.$
 \end{proof}

We note that this improvement is not very useful as $v(\selset,\mathbf C)>w$ is a strong requirement and it is  only satisfied by limited choices of the input.

Moreover, we clearly see how this procedure uses the information of $\mathbf C_i$ for $i\not\in\selset$, which is used to compute $\lambda(\selset,\mathbf C)$, thus violating the requirement (b) in Section \ref{sec:CIreport}.
 Moreover, all input CIs in $\Sigma(w)$  are not strictly decreasing since we need $C_i(\alpha_i)=\Theta_i$ for some $\alpha_i>0$.

\section{Generalized e-CIs and deriving Chernoff CIs}
\label{sec:BatchChernoff}
Extending the notion of e-CI, we can also define a generalized e-CI from a family of e-values that are not only parameterized by a parameter, \(\theta\), but also \(\alpha' \in (0, 1)\) indicating the desired level at which the e-CI is most tight at. Thus, for a family of e-values \(E(\theta, \alpha')\) for \(\theta \in \Theta\) and \(\alpha' \in (0, 1)\), we can define the following generalized e-CI:
\begin{align}
    C^{\alpha'}(\alpha) \coloneqq \left\{\theta \in \Theta: E(\theta, \alpha') < \frac{1}{\alpha}\right\}.
\end{align} Note that \(\alpha'\) is fixed for the e-CI across all choices of \(\alpha \in (0, 1)\). Generalized e-CIs are particularly useful in situations where we have a collection of e-CIs for the same parameter space \(\Theta\), but may be tighter or looser at different values of \(\alpha\). Hence, it is a way of grouping such e-CIs together in one object, and also clarifying that \(\alpha'\) is a fixed parameter for each e-CI in the generalized e-CI and \textit{does not change} with the confidence level \(1 - \alpha\).

Such a situation arises in the case where our e-values are derived from Chernoff bounds. Consider the case where we draw \(n\) i.i.d.\ samples \(X_1, \dots, X_n\) from a distribution that is bounded in \([0, 1]\). Let \(\mu \coloneqq \expect[X_i]\) --- we can derive the following inequality using Hoeffding's lemma:
\begin{align}
    \expect[\exp(\lambda(X_i - \mu)] \leq \exp\left(\frac{\lambda^2}{8}\right).
\end{align}
Let \(\widehat{\mu}_n\) denote the sample mean. Consequently, we can construct the following generalized e-CI from this bound:
\begin{align}
    C^{\alpha'\sdash\mathrm{Hoef}}(\alpha) \coloneqq \left(\widehat{\mu}_n \pm \sqrt{\frac{\log(2 / \alpha)}{2n}} \cdot \frac{\log(2 / \alpha) + \log(2 / \alpha')}{2\sqrt{\log(2 / \alpha')\log(2 / \alpha)}} \right).
\end{align} In contrast, the typical CI constructed from inverting Hoeffding's inequality is as follows:
\begin{align}
    C^{\mathrm{Hoef}}(\alpha) \coloneqq \left(\widehat{\mu}_n \pm \sqrt{\frac{\log(2 / \alpha)}{2n}} \right) = C^{\alpha\sdash\mathrm{Hoef}}(\alpha).
\end{align}

Consider the use of \(C^{\alpha'\sdash\mathrm{Hoef}}\) and \(C^{\mathrm{Hoef}}\) in the e-BY and BY procedures, respectively. As a heuristic, we set \(\alpha' = \delta\). For the selection set \(\selset\), and the same desired level of FCR control \(\delta\), we get the following intervals:
\begin{align}
    C^{\delta\sdash\mathrm{Hoef}}\left(\frac{\delta |\selset|}{K}\right) &= \left(\widehat{\mu}_n \pm \sqrt{\frac{\log\left(\frac{2K}{\delta |\selset|}\right)}{2n}} \cdot \frac{\log\left(\frac{2}{\delta}\right) + \log\left(\frac{2K}{\delta|\selset|}\right)}{2\sqrt{\log\left(\frac{2}{\delta}\right)\log\left(\frac{2K}{\delta |\selset|}\right)}} \right),\\
    C^{\mathrm{Hoef}}\left(\frac{\delta |\selset|}{K\ell_K}\right) &= \left(\widehat{\mu}_n \pm \sqrt{\frac{\log\left(\frac{2K\ell_K}{\delta |\selset| }\right)}{2n}}\right).
\end{align}

This results in the following result about \(\selset\).
\begin{proposition}
\(C^{\delta\sdash\mathrm{Hoef}}\left(\delta |\selset|/K\right)\) is as tight as \(C^{\mathrm{Hoef}}\left(\delta |\selset|/(K\ell_K)\right)\) when
\begin{align}
    |\selset| \geq \frac{K}{\exp\left(2\sqrt{\log\left(\frac{2}{\delta}\right) \log \ell_K}\right)}.
\end{align}
\label{prop:SelsetBound}
\end{proposition}
\begin{proof}
We note that \(C^{\delta\sdash\mathrm{Hoef}}\left(\frac{\delta |\selset|}{K}\right)\) is tighter than \(C^{\mathrm{Hoef}}\left(\frac{\delta |\selset|}{K\ell_K}\right)\) when:
\begin{align}
    \log\left(\frac{2}{\delta}\right) + \log\left(\frac{2K}{\delta | \selset|}\right) \leq 2\sqrt{\log\left(\frac{2}{\delta}\right)\log\left(\frac{2K\ell_K}{\delta | \selset|}\right)}.
\end{align}

This is an inequality on a quadratic expression w.r.t.\ \(\log
|\selset|\). Hence, we can rearrange and solve by quadratic formula. This gets us the following interval for \(\log |\selset|\) where the above inequality holds:
\begin{align}
    \log |\selset| \in \left[\log K \pm 2\sqrt{\log\left(\frac{2}{\delta}\right)\log \ell_K}\right].
\end{align} Since we know \(|\selset|\leq K\), we are only interested in the lower bound on this interval, and we get our desired result from the lower bound.
\end{proof}

In \Cref{prop:SelsetBound}, we see that the lower bound for \(|\selset|\), as a proportion of \(K\), decreases as \(K\) increases. Thus, \(C^{\delta\sdash\mathrm{Hoef}}\left(\delta |\selset|/K\right)\) is tighter than \(C^{\mathrm{Hoef}}\left(\delta |\selset|/(K\ell_K)\right)\) for an increasing proportion of possible values of \(|\selset|\) as \(K\) grows. We see this advantage of the e-BY procedure over the BY procedure reflected in our simulation results in \Cref{sec:Simulations} --- the CIs output by the e-BY procedure are tighter than the CIs of the BY procedure (under no assumptions on selection rule or dependence) as \(K\) grows.

\section{FCR and other measures of statistical validity}
\label{sec:CoverageTypes}

We will elaborate on the points made in \Cref{sec:Intro} about the other error metrics one may consider applicable to the post-selection inference problem.

\paragraph{Conditional coverage requires knowing about selection rule a priori.} Before discussing FCR control, we can first consider if a coverage guarantee is possible when we condition on a parameter being selected. In the majority of this paper, we have assumed that each \(C_i\) is a \textit{marginal} CI, i.e.,\ for each \(i \in [K]\), \(\prob{\theta_i^* \in C_i(\alpha)} \geq 1 - \alpha\) for every \(\alpha \in (0, 1)\). For conditional coverage, we want a confidence level (\(1-\alpha_i\)), for each \(i \in [K]\), such that the resulting CI, \(C_i(\alpha_i)\), ensures \(\prob{\theta_i^* \in C_i(\alpha_i) \mid i \in \selset} \geq 1 - \delta\). However, BY noted that no such correction can be made without knowing the selection rule. To illustrate this, consider the following toy example. Assume that the parameters being estimated are means, and the data distribution for each parameter is Gaussian with variance 1, i.e.,\ \(X_i \sim \mathcal{N}(\theta_i^*, 1)\). Then, the CI for the \(i\)th parameter is \(C_i(\alpha) = [X_i - z_{\alpha_i / 2}, X_i + z_{\alpha_i / 2}]\) where \(z_\alpha\) is the \((1 - \alpha)\)-quantile of the standard normal distribution. Now, select the \(i\)th parameter only if \(C_i(\alpha_i)\), the corrected CI, covers solely positive values. If \(\theta_i^* \leq 0\), the conditional coverage is 0. Thus, no guarantees can be made about conditional coverage without knowing or constraining the selection rule.

Note that if we consider a constraint on the joint probability of a parameter being selected and miscovered instead, i.e.,\  \(\prob{\theta_i^* \notin C_i(\alpha_i), i \in \selset} \leq \alpha \), then the CI guarantee itself ensures that this constraint is satisfied, i.e.,\ \(\prob{\theta_i^* \notin C_i(\alpha_i)} \leq \alpha\). The fact that we may bound this joint probability is what motivates the possibility of the use of FCR as a error metric we can control.

In addition, the refutation of conditional coverage guarantees is intended for the general setting, where there is no limitation of the choice of the selection rule. Thus, methods that do provide conditional coverage guarantees must depend on the specific selection rule being used. \hbox{\citet{zhong_biasreduced_estimators_2008}} and \citet{weinstein_selection_adjusted_2013c} follow this line of inquiry and derive explicit CIs that have the conditional coverage guarantee under specific selection rules and assumptions on the families of parameters. The difficulty in this approach is that the results provided are narrowly applicable to the predetermined selection rule and data distribution families --- it requires the user to derive new conditional CIs for different combinations of selection rules and family of distributions. \citet{weinstein_online_control_2020} come to a similar conclusion about the difficulty of using conditional CIs in an online version of the post-selection inference problem, and also recommend the marginal CI approach used in both BY and e-BY. Further, they note that an additional deficit of the conditional coverage approach is that the conditionally valid CIs are not necessarily ``consistent'' with the selection procedure. To refer to the toy example, the conditional CIs for a selection rule that aims to only select parameters that have CIs that are completely positive must also include negatives (otherwise the toy example still holds). In this sense, the produced conditional CIs are inconsistent with the goal of the selection rule. Hence, FCR control circumvents these challenges, at the cost of providing a less powerful error guarantee.

\paragraph{FCR control results in smaller CIs than simultaneous coverage.} Another type of error control to consider is a simultaneous guarantee over the CIs of all selected parameters, i.e.,\ \(\prob{\forall i \in \selset: \theta_i^* \in C_i(\alpha_i)} \geq 1 - \delta\), for some choices of \(\alpha_i \in [0, 1]\) for each \(i \in \selset\). This error is called the \textit{simultaneous over selected (SoS)} criterion in \citet{benjamini_confidence_intervals_2019} --- they provide methods for controlling the SoS for a few, specific selection rules. Note that control of the SoS criterion implies control of the FCR, but not necessarily the other way around. Hence, SoS control is strictly stronger than FCR control. The only known way to achieve this simultaneous coverage for arbitrary selection rules, however, is to make a Bonferroni correction and ensure simultaneous coverage of all the CIs, i.e.,\ for both selected and non-selected parameters. If all confidence levels are equivalent, this implies \(\alpha_i = \delta / K\). This is a much more conservative level of correction than the BY and e-BY procedures.

\end{document}